\documentclass[11pt]{amsart}
\usepackage{amsmath,amsthm,amsfonts,amssymb,verbatim,eucal, mathrsfs}
\usepackage{tikz-cd}
\usepackage[all]{xy} 
\usepackage{extarrows, mathtools}
\usepackage[outdir=./]{epstopdf}

\usepackage{hyperref}
\hypersetup{colorlinks=true,citecolor=blue!70!black,linkcolor=red!60!black,linktocpage=true}
\usepackage{aliascnt}
\usepackage[capitalise,nameinlink]{cleveref}
\crefname{enumi}{part}{parts}
\Crefname{diagram}{Diagram}{Diagrams}
\creflabelformat{diagram}{#2(#1)#3}

\numberwithin{equation}{section}

\newtheorem{thm}[equation]{Theorem} 

\newtheorem{lemma}[equation]{Lemma} 
\newtheorem{cor}[equation]{Corollary}
\newtheorem{example}[equation]{Example}
\newtheorem{remark}[equation]{Remark}
\newtheorem{definition}[equation]{Definition}

\DeclareRobustCommand\longtwoheadrightarrow{\relbar\joinrel\twoheadrightarrow}

\DeclareMathOperator{\Ima}{Im}
\DeclareMathOperator{\ima}{Im}
\DeclareMathOperator{\Ker}{Ker}
\DeclareMathOperator{\sgn}{sgn}

\newcommand{\pr}{{\rm{pr}}}
\newcommand{\inc}{{\rm{in}}}

\newcommand{\GL}{\text{GL}}
\newcommand{\igamma}{\gamma^{-1}}
\newcommand{\R}{\mathscr{R}}

\newcommand{\sd}{_{\DOT}}
\newcommand{\B}{{\bf B}} 
\newcommand{\BB}{\mathbb{B}} 
\DeclareMathOperator{\AW}{AW}
\DeclareMathOperator{\EZ}{EZ}
\DeclareMathOperator{\BAW}{{\AW}}
\DeclareMathOperator{\BEZ}{{\EZ}}
\DeclareMathOperator{\AWt}{\AW^{\tau}}
\DeclareMathOperator{\EZt}{{\EZ^{\tau}}}
\DeclareMathOperator{\AWtB}{ \AW^{\tau}_{\!\B}}
\DeclareMathOperator{\EZtB}{  \EZ^{\tau}_{ \B} }
\DeclareMathOperator{\AWtBB}{ \AW^{\tau}_{\!\BB}}
\DeclareMathOperator{\EZtBB}{  \EZ^{\tau}_{ \BB} }

\newcommand*\xbar[1]{%
  \hbox{%
    \vbox{%
      \hrule height 0.5pt 
      \kern0.2ex
      \hbox{%
        \kern-0.0em
         \ensuremath{#1 }
        \kern-.1em
              }%
    }\hphantom{.}
  }%
}

\newcommand{\itau}{{\tau^{-1}}}
\newcommand{\s}{\sigma}
\newcommand{\ttp}{R \ot_{\tau} S}
\newcommand{\ttpp}{(\ttp)}
\newcommand{\ttps}{R\,\ot_{\tau} S}
\newcommand{\hooklongrightarrow}{\lhook\joinrel\longrightarrow} 

\newcommand{\DOT}{\setlength{\unitlength}{1.2pt}\begin{picture}(2.5,2) 
               (1,1)\put(2.5,2.5){\circle*{2}}\end{picture}}

\newcommand{\coh}{{\rm H}}

\newcommand{\N}{{\mathbb N}}

\renewcommand{\ker}{\mbox{\rm Ker\,}}

\newcommand{\ot}{\otimes}
\newcommand{\ott}{\otimes_{\tau}}

\newcommand{\Ho}{H}

\newcommand{\CC}{\mathbb{C}}

\newcommand{\del}{\partial}

\newcommand{\Wedge}{\textstyle\bigwedge}

\newcommand{\Ydiff}{\partial}

\begin{document}

\begin{abstract} 
Alexander-Whitney 
and Eilenberg-Zilber maps traditionally convert
between the
tensor product of standard resolutions and the standard
resolution of a tensor product of algebras. 
We examine Alexander-Whitney and Eilenberg-Zilber maps for twisted
tensor product algebras, which include
(quantum) universal enveloping algebras, 
skew group algebras, smash and crossed products,
Ore extensions, and algebras with triangular decompositions, for example.
These maps convert between the twist of standard resolutions
and the standard resolution of a twist.  We extend these 
to chain maps to and from twists of other resolutions.
This allows one to transfer
homological information between various resolutions of algebras
and to expedite results in
deformation theory.
\end{abstract}

\title[Twisted tensor products:
Alexander-Whitney and Eilenberg-Zilber maps]
      {Twisted tensor products:\\  Alexander-Whitney and Eilenberg-Zilber 
      maps}

 \date{September 12, 2026} 



\author{A.\ V.\ Shepler}
\address{Department of Mathematics, University of North Texas,
Denton, TX 76203, USA}
\email{ashepler@unt.edu}
\author{S.\  Witherspoon}
\address{Department of Mathematics\\Texas A\&M University\\
College Station, TX 77843, USA}\email{sjw@tamu.edu}

\dedicatory{Dedicated to J.\ Peter May on the occasion of his 85th birthday.}
\maketitle

\section{Introduction}
Various algebraic structures in homology and cohomology are defined
in terms of the bar (i.e., standard) resolution of an algebra, but the
infinite bar resolution can be taxing for direct computation.  The terms grow in
complexity as the homological degree increases, often limiting extraction
of concrete information. One seeks a more convenient
resolution to study homology or cohomology, but manageable resolutions can come
at a price:
homological constructions like the cup product and graded Lie bracket
often do not relocate easily to alternate resolutions.
Explicit chain maps to transfer information
between resolutions are often lacking,
which limits fruitful illumination provided by homological methods.

Traditionally,
Alexander-Whitney 
and Eilenberg-Zilber maps convert
between the chain complex
of a product of spaces $R\times S$ and the
tensor product of the chain complexes of the individual spaces $R$ and $S$, providing 
an isomorphism 
when homology groups are free, $\coh(R\times S)\cong
\coh(R)\otimes \coh(S)$.
Composing the maps in one order gives the identity on chains,
$\AW\circ\EZ = 1$,
while composing in the reverse order
gives a map only homotopic to the identity, $\EZ\circ\AW \cong 1$,
for
$$
\begin{aligned}
  &\AW:& \text{Chains}(R\times S)&\longrightarrow \text{Chains}(R)\ot \text{Chains}(S), 
  \\
  &\EZ:& \text{Chains}(R)\otimes \text{Chains}(S) &\longrightarrow \text{Chains}(R\times S) 
  \, . 
\end{aligned}
$$
The traditional Eilenberg-Zilber map
(also called the ``shuffle product'' or ``simplicial cross product'')
was originally
defined by Eilenberg and Mac Lane \cite[$\S$5]{EilenbergMacLane}
to give the isomorphism in
the Eilenberg-Zilber Theorem~\cite{EZ53};
see also \cite{EZ50}, \cite[\S29.7]{May}, 
\cite[Exercise~8.6.5]{Weibel},
\cite[\S3.B, p.~277]{Hatcher}, and~\cite[Chapter~5.3]{AguiarMahajan},
as well as~\cite[Chapter VIII, Theorems 8.1, 8.5, 8.8]{MacLane} 
for simplicial modules.

A version of this theory for
the tensor product of algebras $R$ and $S$ is described
in Mac Lane~\cite[Chapter X, \S7 and \S12]{MacLane}, Chevalley and
Eilenberg~\cite[Chapter XI, \S6]{CE}, May~\cite{May},
and Loday~\cite[\S4.2]{Loday}.
Algebraic versions of the above maps likewise provide a
strong deformation retraction of the reduced bar
resolution 
of the tensor product 
onto the tensor product of the reduced bar resolutions:
$$
\begin{aligned} 
  &\AW:& \BB_{R\otimes S}&\longrightarrow \BB_R\ot \BB_S, 
  \\
  &\EZ:& \BB_R\otimes \BB_S &\longrightarrow \BB_{R\otimes S} 
  \, . 
\end{aligned}
$$
Here, $\BB_A$ denotes the reduced (normalized) bar resolution  
of an algebra $A$, so
$(\BB_A)_n=A\ot \bar{A}^{\ot n}\ot A$  
for $\bar A=A/k\cdot 1_A$.  
Le and Zhou~\cite{LZ} described the Gerstenhaber structure on
Hochschild cohomology using these maps.

We consider the analogous theory and applications for {\em twisted
  tensor product algebras},
which are 
abundant
in representation theory and combinatorics.
A twisted tensor product is simply any 
algebra isomorphic to $A\ot B$ as a vector space for two
subalgebras
$A$ and $B$ under the canonical inclusion maps.
This large class of noncommutative algebras
includes, for example,
universal enveloping algebras and quantum versions (in the
supersolvable case),
quantum groups with triangular decompositions,
skew group algebras, 
smash products with Hopf algebras,
semi-direct and crossed products,
Jordan planes,
some Sridharan enveloping algebras,
quantum complete intersections,
Ore extensions, some Sklyanin
algebras, Weyl algebras, quantum/skew versions of polynomial rings,
quantum Schubert cell algebras, 
and braided tensor products defined by $R$-matrices.

We show how to use
Alexander-Whitney and Eilenberg-Zilber maps
for twisted tensor products $R\ott S$
to transfer between resolutions carrying homological information.
These are chain maps that convert between
the reduced bar resolution $\BB_{R\ott S}$ on one hand
and a twisted product resolution $\BB_{R}\ott\,\BB_S $
on the other hand, see Guccione and Guccione~\cite{GG}
and \cref{AW-EZforreducedbar}:
$$
\begin{aligned} 
  &\AW^{\tau}:& \BB_{R\ott S}&\longrightarrow \BB_R\ott \BB_S, 
  \\
  &\EZ^{\tau}:& \BB_R\ott \BB_S &\longrightarrow \BB_{R\ott S} 
  \, . 
\end{aligned}
$$
See also~\cite{GNW}~and~\cite{SW-grouptwistedAWEZ} 
for these maps in special cases of twisted
tensor products involving group actions or  bicharacters of grading groups.

In practice, one often finds compatible resolutions $C\sd$ and $D\sd$ of
the two individual algebras $R$ and $S$
which are easier to work with than bar resolutions
and allow for explicit computations.
For example, one might desire a periodic resolution for
a cyclic group (see \cite{LawsonShepler}, for example),
the Chevalley-Eilenberg
resolution for a universal enveloping algebra (e.g., see \cite{Weibel}),
or a Koszul resolution for a Koszul algebra.
We develop tools
over arbitrary fields
(see \cref{barconversion} and \cref{conversionspecial}, for example)
for converting 
from the bar resolution of the twisted tensor product $R\ott S$
to a twist $C\sd\ott D\sd$ of individual resolutions for $R$ and $S$,
resolutions that potentially lend themselves to a 
concrete determination of homology or cohomology.

As an application, we consider Hopf algebras acting on Koszul
algebras.
Deformations of corresponding smash product algebras
(twisted tensor products)
generalize the graded affine Hecke algebras of Lusztig
\cite{Lusztig88, Lusztig89}, the Drinfeld Hecke algebras of
Drinfeld~\cite{Drinfeld}
(see also \cite{Gordon} and 
\cite{RamShepler}), the symplectic reflection algebras
of Etingof and Ginzburg~\cite{EG}, and
the deformations of Kleinian singularities 
of Crawley-Boevey and Holland \cite{CBH}. 
These arise from groups acting on polynomial rings;
more general Hopf algebra actions appear for example
in~\cite{EGG, Flake19, FlakeSahi, GanKhare, KhareTikaradze, LSS, WW}.
In~\cite{SW-DeformationTheoryHopfActions},  
we consider deformations of smash product algebras $R\# H$
in the general case
for Hopf algebras $H$ acting on Koszul algebras $R$. 
We construct ``Hopf-Koszul Hecke algebras"
using conversion chain maps established here
in
\cref{LastThm}
to transfer 
between resolutions even when the Hopf algebra is not semisimple. 

As a running example throughout, we make explicit constructions for finite groups acting on polynomial rings. See
\cref{GroupUnshuffle,GroupExampleAWEZmaps,SymmetricAlgebraGroup,ChainMapsGroupActingOnPolyRing}
and the special case at the end of~\cref{sec:conversion}.

Note that authors sometimes use the term 
``twisted tensor product" when 
twisting 
by a cocycle or cochain 
(see~\cite{AlvarezArmarioFrauReal,Brown1959,Brown1964} for example).
This notion of twisting is generally different from that 
considered here.  

\subsection*{Outline}
In \cref{sec:NW1}, we recall basic notations of twisted tensor
products and review
the construction of the twisted product resolution.
We specialize to twists of bar resolutions
in \cref{twistedproductresolutionbars}.
We construct twisted Alexander-Whitney and Eilenberg-Zilber maps in
\cref{sec:AW-EZ} and recall the special case of finite groups acting
on polynomial rings.
In  \cref{sec:chainmaps,sec:chainmapsonreduced},
for completeness, we verify these are chain maps
with composition
$\EZ^{\tau} \circ \AW^{\tau} = 1$, the identity map,
on the reduced bar resolutions.
We show in \cref{sec:compatiblechainmaps,sec:conversion} how to use these maps to convert
between other resolutions for twisted tensor products and again 
illustrate with  the case
of groups acting on polynomial rings.
We turn to the actions of Hopf algebras generally 
in Sections~\ref{sec:smashproducts} and~\ref{sec:Hopf-Koszul},
showing how to convert between
resolutions for smash product algebras,
with particular focus on
actions on Koszul algebras and their Koszul resolutions.

\subsection*{Notation and Conventions}
We let $k$ be an arbitrary field with char$(k)\neq 2$ throughout.
We take all tensor products over $k$ unless otherwise indicated,
$\otimes=\otimes_k$.  The identity element of $k$
is identified with the identity element of each algebra $A$ we encounter,
although we often write $k 1_A$ instead of $k$ for clarity. 
We assume all algebras are associative $k$-algebras
unless otherwise indicated.
For any algebra $A$, we write $m_A$ for the multiplication map 
$A\ot A\rightarrow A$ and 
view any $A$-bimodule as
a left $A^e$-module for $A^e=A\ot A^{\text{op}}$.
We also follow
the standard {\em Koszul sign convention}:
For graded vector spaces $P, P', Q, Q'$ over $k$ and graded maps
$f: P\rightarrow Q$ and $f': P'\rightarrow Q'$,
the map $$f\ot f': P\ot P'\longrightarrow Q\ot Q'$$ is defined
for homogeneous $p$ in $P$ and $p'$ in $P'$ by
\begin{equation}
  \label{signconvention}
(f\ot f')(p\ot p')= (-1)^{(\deg p)(\deg f')}\big(f(p)\ot f'(p')\big)
\, .
\end{equation}

\section{Twisted product resolutions}\label{sec:NW1}

We explain how to twist together resolutions of individual algebras
to give new resolutions of twisted tensor product algebras.

\subsection*{Twisted tensor products}
A {\em twisted tensor product} is an algebra
generated by two subalgebras $R$ and $S$ 
that is isomorphic as
a vector space to $R\ot S$ via the canonical map $r\ot s \mapsto rs$.
We may construct a 
{\em twisted tensor product} $R\ott S$ 
from a pair of $k$-algebras $R$ and $S$
as the $k$-vector space $R\ot S$
with  associative multiplication
determined by a twisting map $\tau: S \ot R \to R \ot S$
that commutes with the multiplication in $R$ and $S$.
Specifically, 
a bijective $k$-linear  map $$\tau: S \ot R \longrightarrow R \ot S$$
is a {\em twisting map}
when
$
\tau(1_S \ot r) = r\ot 1_S  \ \mbox{ and } \ 
\tau(s \ot 1_R) = 1_R \ot s 
$ 
for all $r\in R$ and $s\in S$, and
there is an equality of compositions
mapping $S\ot S\ot R\ot R$ to $R\ot S$:
\begin{equation}\label{eqn:taumm}
\tau \ (m^{}_{S} \ot m^{}_{R}) = 
(m^{}_{R} \ot m^{}_{S}) \ (1 \ot \tau \ot 1) 
\ (\tau \otimes \tau) \ (1 \ot \tau \ot 1)
\, .
\end{equation}
The multiplication on $R\ott S$ is given by
$(m^{}_{R}\ot m^{}_{S}) (1_R\ot \tau\ot 
1_S)$;
see \cite[Proposition/Definition~2.3]{CSV} for details.

Examples of twisted tensor products include
skew group algebras, smash product algebras,
Ore extensions (including universal enveloping algebras
of finite dimensional supersolvable Lie algebras), braided
products arising from $R$-matrices, and $2$-cocycle twists of Hopf algebras. 
See Conner and Goetz~\cite{ConnerGoetz} for additional examples.

\vspace{2ex}

\begin{remark}\label{TwistingDefDiagram}
  {\em
    By definition, a twisting map satisfies this commutative diagram:
    \vspace{1ex}
\begin{footnotesize}
\begin{equation*}
\entrymodifiers={+!!<0pt,\fontdimen22\textfont2>}
\xymatrixcolsep{9ex}
\xymatrixrowsep{7ex}
\xymatrix{
S\ot S \ot R\ot  R
\ar[dr]_{\ m^{}_{S} \ot m^{}_{R}\ }
\ar[r]^{\ 1\ot\tau\ot 1 \ }
&S\ot R\ot S \ot R
\ar[r]^{\tau\ot\tau} 
& R\ot S \ot R \ot S
\ar[r]^{1\ot \tau\ot 1}
&
R\ot R \ot S \ot S
\ar[dl]^{\ m^{}_{R}\ot m^{}_{S}}
&
\\
& S\ot R 
\ar[r]_{\ \tau\  }
&
R\ot S
&&\\
.
}
\end{equation*}
\end{footnotesize}%
  }
\end{remark}

\subsection*{Graded twisted products}
For $\N$-graded algebras $R$ and $S$, we take the standard $\N$-grading 
on the vector spaces $R\ot S$ and $S\ot R$.
The twisted product $\ttp$ is then $\N$-graded
when the twisting map $\tau$ is graded,
and we say $\tau$
is {\em strongly graded} when
$\tau(S_j\ot R_i)\subset R_i\ot S_j$ for all $i,j$
(see Conner and Goetz~\cite{ConnerGoetz}). 
(Note that some authors leave off the adjective {\em strongly}.)

\subsection*{Resolutions compatible with twisting}
We fix algebras $R$ and $S$
and twisting map $\tau:S\ot R\rightarrow R\ot S$
and consider bimodule resolutions
$$
\begin{aligned}
C\sd:\quad \dots \longrightarrow & \ \, C_2 \longrightarrow C_1 \longrightarrow C_0 \longrightarrow 0 
\ \ \text{ of } R, \quad \text{ and}\\
D\sd:\quad \dots \longrightarrow & \ D_2 \longrightarrow D_1 \longrightarrow D_0 \longrightarrow 0 
\ \ \text{ of } S \text{,}
\end{aligned}
$$
with differentials $\partial^{}_C$ and $\partial^{}_D$, respectively.
Let $\rho_{_C}:R\ot C\sd \ot R\rightarrow C\sd$
and $\rho_{_D}:S\ot D\sd\ot S\rightarrow D\sd$
record the bimodule structures.
Recall that we say a chain map $\psi$ between resolutions $P\sd$ and $P'\sd$ of  
an algebra $A$ {\em lifts the identity map} on $A$ 
when $\psi$ extends to a chain map of the complexes
augmented
by $A$ in degree $-1$ with $\psi_{_{-1}}=1_A$.
We are interested in the case when resolutions $C\sd$ and $D\sd$ are compatible
with twisting (see~\cite[Definition~2.17]{SW-twisted}):

\vspace{2ex}

\begin{definition}\label{CompatibleResolutionsDef}
  {\em
    An $R$-bimodule resolution
        $C\sd$ of $R$
    is  {\em compatible} with the twisting map
    $\tau:S\ot R\rightarrow R \ot S$
 when there is a bijective chain map
$$
\tau_{_C}: S\ot C\sd \longrightarrow C\sd \ot S
$$
lifting $\tau$
that commutes with the bimodule structure 
of $C\sd$ and multiplication in $S$.
Likewise, we say an $S$-bimodule resolution $D\sd$ of $S$
is {\em compatible with $\tau$} 
when it is compatible with
$\tau^{-1}:R\ot S \rightarrow S\ot R$,
and we abbreviate (when there is no risk of confusion)
the chain map by
$$\tau^{}_{_D}=((\tau^{-1})_{_D})^{-1}:
D\sd\ot R\longrightarrow R\ot D\sd\ . $$
}
\end{definition} 

\vspace{2ex}

Observe that the right-left distinction in this definition
is merely cosmetic: For $R=S$ and $C\sd=D\sd$,
a bijective chain map $R\ot C\sd\rightarrow C\sd \ot R$
lifting $\tau$ commutes with the bimodule structure
of $C\sd$ and multiplication in $R$ precisely when the inverse chain map
$ C\sd \ot R \rightarrow R\ot C\sd$ does.

\vspace{2ex}

\begin{remark}{
\label{CompatibleResolution}
\em
Specifically, the $R$-bimodule chain complex $C\sd$ of $R$
is  {\em compatible} with the twisting map $\tau: S\ot R\rightarrow R\ot S$
precisely when 
there are bijective $k$-linear maps $\tau^{}_{_{C_i}}: S\ot C_i \rightarrow C_i\ot S$
for which the following diagrams commute:
\begin{small}
$$
\entrymodifiers={+!!<0pt,\fontdimen22\textfont2>}
\xymatrixcolsep{8ex}
\xymatrixrowsep{5ex}
\xymatrix{ 
\cdots \ar[r] & 
{S} \ot {C_1}
\ar[r]^{1\ot \del^{}_{C}}
\ar[d]_{\tau^{}_{_{C_1}}}
&  
{S} \ot {C_0}
\ar[r]^{1\ot \del^{}_C}
  \ar[d]_{\tau^{}_{_{C_0}}}
  & S\ot R
  \ar[r]
 \ar[d]_{\tau} 
&
  0
  \\
  \cdots \ar[r] 
&
{C_1} \ot {S}   \ar[r]_{\del^{}_C\ot 1} 
&
{C_0} \ot {S}
  \ar[r]_{\del^{}_C\ot 1} 
   & R \ot {S}
  \ar[r]
    & 0
}
$$
\end{small}
and
\begin{small}
\begin{equation*}
\entrymodifiers={+!!<0pt,\fontdimen22\textfont2>}
\xymatrixcolsep{2ex}
\xymatrixrowsep{5ex}
\xymatrix{
&&
S\ot C\sd\ot S 
\ar[drr]^{\ \tau_{_C}\ot 1}
&&
\\
S\ot S\ot C\sd
\ar[urr]^{\ 1\ot \tau_{_C}\ \ }
\ar[dr]_{\ m^{}_S\ot 1\ \ }
&&&&  
C\sd \ot S \ot S
\ar[dl]^{\ \ 1\ot m^{}_S}
\\ 
&S\ot C\sd
\ar[rr]_{\tau_{_C}} 
&&
C\sd\ot S
&
\\
S\ot R\ot C\sd\ot R
\ar[dr]_{\tau\ot 1\ot 1\ \ \ }
\ar[ur]^{1\ot\rho_{_C}\ \ \ }
&& &&
R\ot C\sd\ot R\ot S
\ar[ul]_{\ \ \ \rho_{_{C}}\ot 1}
\\
&
R\ot S\ot C\sd\ot R
\ar[rr]_{\ 1\ot\, \tau_{_C}\ot 1\ }
&&
R\ot C\sd\ot S\ot R
\ar[ur]_{\ \ \ 1\ot 1\ot \tau}
&
.
}
\end{equation*}
\end{small}%
A similar diagram captures compatibility of an $S$-bimodule resolution $D\sd$ of $S$, see \cite{SW-twisted}.
  }
\end{remark}


\subsection*{Twisted product resolution as a total complex}
For bimodule resolutions $C\sd$ and $D\sd$ of $R$ and $S$,
respectively,
compatible with the twisting map $\tau:S\ot  
R\rightarrow R\ot S$,  
the {\em twisted product resolution} $X\sd=C\sd \ot_{\tau} D\sd$ 
of the algebra $\ttp$ is
the total complex of $C\sd \ot D\sd$
with a particular bimodule action of $\ttp$:
As a complex of vector spaces,
\begin{equation}\label{cx-X}
X_n = \bigoplus_{i+j=n} C_i \ot D_j
\, \qquad\text{ for } n\geq 0 \, 
\end{equation}
with differential
$\del^{}_X=\del^{}_C\ot 1^{}_D + 1^{}_C\ot \del^{}_D$
(see (\ref{signconvention})):
\begin{small}
$$
\entrymodifiers={+!!<0pt,\fontdimen22\textfont2>}
\xymatrixcolsep{6ex}
\xymatrixrowsep{4ex}
\xymatrix{
\vdots \ar[d] & &
\vdots \ar[d] & &
\vdots \ar[d] \\
{{C_0} \ot {D_2} }\ar[d]^{1\ot \del^{}_D} & &
{{C_1}\ot {D_2}} \ar[ll]_{\del^{}_C\ot 1} \ar[d]^{1\ot \del^{}_D}  & &
{{C_2}\ot {D_2}} \ar[ll]_{\del^{}_C\ot 1}
\ar[d]^{1\ot \del^{}_D}  &
{\cdots}
\ar[l]  \\
   {C_0} \ot {D_1} \ar[d]^{1\ot \del^{}_D}
   & &
   {C_1} \ot {D_1} \ar[ll]_{\del^{}_C\ot 1}
   \ar[d]^{1\ot \del^{}_D} & &
{C_2} \ot {D_1} \ar[ll]_{\del^{}_C\ot 1} \ar[d]^{1\ot \del^{}_D} &
{\cdots}
\ar[l]
\rule{0ex}{6ex}
\\
{C_0} \ot  {D_0}  & &
{C_1} \ot {D_0} \ar[ll]_{\del^{}_C\ot 1} & &
{C_2}\ot {D_0} \ar[ll]_{\del^{}_C\ot 1} &
\cdots\ar[l]
}
$$
\end{small}

\subsection*{Bimodule structure on the total complex}
We imbue each $X_n$ with the structure
of an $(\ttp)$-bimodule using compatibility maps:
\vspace{2ex}
  \begin{small} 
    \begin{equation}
    \label{bimodstructure}
  \begin{tikzcd}[row sep=7ex, column sep=8ex, text height=2ex, text depth=.25ex, ampersand replacement=\&]
   \&[-3em]
   (R \ott  S) \ot C_i \ot  D_j\ot  (R \ott  S)
    \arrow[swap, bend right=35, dashed]{d}{\text{bimod structure}\ }
   \arrow{rr}{ ^{ 1\ot\, \tau^{}_{_C}\ot\, \tau^{}_{_D} \ot 1} }
   \&\&
   R \ot  C_i \ot  S \ot  R \ot  D_j \ot S 
   \arrow[]{d}{1\ot 1\ot\, \tau\, \ot 1\ot 1 }
   \\
\&
    C_i\ot D_j 
    \&\&
         R \ot  C_i \ot R \ot  S \ot  D_j \ot S \, ,
 \arrow{ll}{ \rho_{_C}\ot\, \rho_{_D}}
    \end{tikzcd}
  \end{equation}
\end{small}%
\vspace{2ex}
i.e.,
with left-right $(R\ott S)$-action given by
the composition of $k$-vector space maps
$$ 
(\rho_{_C}\ot\, \rho_{_D})\ 
(1\ot 1\ot \tau \ot 1\ot 1 )
\ 
(1\ot\, \tau^{}_{_C}\ot\, \tau^{}_{_D}\, \ot 1)\, .
$$

\subsection*{Twisted product complex as a resolution}
The total complex
$X\sd$
with this $\ttpp$-bimodule structure is a resolution of $\ttp$:
The complex
\begin{equation}\label{twistedproductresolution}
\cdots\xrightarrow{ \hphantom{xxxxx}}
X_2\xrightarrow{ \hphantom{xxxxx}}
X_1 \xrightarrow{ \hphantom{xxxxx}}
X_0 \xrightarrow{ \hphantom{xxxxx}}
\ttp \longrightarrow 0\, 
\end{equation} 
is an exact 
sequence of $(\ttp)$-bimodules by 
\cite[\S3]{SW-twisted} (also see ~\cite[\S4]{quad}).
Consequently, if the $(\ttp)$-bimodules $X_n$ are all projective as 
$(\ttp)^e$-modules, then $X\sd$ is a projective resolution of $\ttp$
(see~\cite{SW-twisted}).

\vspace{2ex}

\begin{remark}\label{examples}{\em 
  The twisted product complex $X\sd$ gives a convenient projective
  resolution of $R\ott S$ in many common settings.
  For example,
 the Koszul resolution $K\sd$ of a graded Koszul algebra $R$ and
  the bar (or reduced bar) resolution $\B_S$
  of  a graded algebra $S$ (see \cref{twistedproductresolutionbars})
are both compatible with a strongly graded twisting map $\tau: S\ot
R\rightarrow R\ot S$,
and they give a free bimodule 
resolution $K\sd\ott \B_S$
of $R\ott S$.  See~\cite[Proposition 2.20(ii)]{SW-twisted} and
\cref{BarCompatible}.
This setting includes the case of a Hopf algebra $H$ acting on a graded Koszul algebra;
also see \cref{BarIsComoduleRes}
and (\ref{KoszulIsCompatible}) below.
In particular, for finite groups acting on polynomial rings,
see
\cref{GroupUnshuffle,SymmetricAlgebraGroup,ChainMapsGroupActingOnPolyRing}
and the end of~\cref{sec:conversion}.

Another example arises from iterated Ore extensions
(see \cite{BancroftShepler} and \cite{SW-twisted})
with relations that set quantum/skew commutators to lower order terms.
The twisted product resolution construction in this setting
generalizes the Cartan-Eilenberg complex
(resolving 
universal enveloping algebras as bimodules)
for finite-dimensional 
Lie algebras
solvable over $\CC$ 
(or supersolvable over $k$ when $k$ is not algebraically
closed or $\text{char}\ k\neq 0$).
It also gives standard resolutions for Weyl algebras, the
quantum Koszul complex of Wambst~\cite{Wambst1993},  and resolutions
of quantum Schubert cell algebras and some
Sridharan enveloping algebras, for example. 
}
\end{remark}

\section{Twisting bar resolutions}
\label{twistedproductresolutionbars}
We consider the twisted product resolution of 
bar complexes here for a pair of algebras $R$ and $S$
with twisting map $\tau:S\ot R\rightarrow R\ot S$
before turning to general resolutions later.

\subsection*{Bar resolutions}
The {\em bar resolution} $\B_A = (\B_A)\sd$
and {\em reduced} (i.e., {\em normalized}) {\em bar resolution} $\BB_A = (\BB_A)\sd$
of an algebra $A$ are respectively given by
$$
\begin{aligned}
  (\B_A)\sd\ :\ \ 
  & \cdots \longrightarrow (\B_A)_2 \longrightarrow (\B_A)_1 \longrightarrow A\ot
  A
  \xrightarrow{m_{_A}} 
A \longrightarrow 0 \qquad && \text{ (bar) }
\\
(\BB_A)\sd\ :\ \ 
&\cdots \longrightarrow (\BB_A)_2 \longrightarrow \, (\BB_A)_1 \longrightarrow A\ot A
\xrightarrow{m_{_A}} 
A \longrightarrow 0
&& \text{ (reduced bar) }
\end{aligned}
$$
 with 
$$
\begin{aligned}
  (\B_A)_n = A \ot A^{\, \ot n}\ot A\,
  \qquad\text{ and }\qquad
  (\BB_A)_n = A \ot \bar{A}^{\, \ot n}\ot A\,
  \qquad\text{ for } n> 0
\end{aligned}
$$
for $\bar{A}=A/(k\cdot 1_A)$
with $(\B_A)_0=(\BB_A)_0=A\ot A$
and differential defined by 
\begin{equation}\label{eqn:bar-diff}
   d_n(a_0\ot\cdots \ot a_{n+1})
   = \sum_{i=0}^{n} (-1)^i  a_0\ot\cdots\ot a_{i-1}\ot a_ia_{i+1}\ot a_{i+2}\ot \cdots
   \ot a_{n+1} 
\end{equation}
on both $\B_A$ and $\BB_A$.
We identify $\bar{A}$ with a fixed vector space complement
to $k1_A$ in $A$ by choosing a section to the projection map
$\pr^{}_{\bar A}: A\twoheadrightarrow A/ k 1_A$
so that 
\begin{equation}\label{eqn:bar-id}
  A=\bar{A}\oplus k1_A
\end{equation}
and the multiplication in $\bar{A}$ is given by the multiplication
in $A$ followed by projection to $\bar{A}$.
With this choice, we identify each bimodule component 
$(\BB_A)_n$ with a vector space direct summand 
of $(\B_A)_n$ via the  
inclusion map $\inc^{}_{\BB_A}: (\BB_A)_n \hookrightarrow (\B_A)_n$,
so that
\begin{equation}\label{DifferentialReduced}
      d^{}_{{\BB_A}} =
    \pr^{}_{\BB_A}\, d^{}_{\B_A} \, \inc^{}_{\BB_A}
 \end{equation}
  for $d^{}_{\B_A}$, $d^{}_{\BB_A}$ the differentials on $\B_A$, $\BB_A$, respectively,
and $\pr^{}_{\BB_A}:\B_A\twoheadrightarrow \BB_A$
the projection map $1_A\ot \pr^{}_{\bar{A}}\ot\cdots\ot
\pr^{}_{\bar{A}}\ot 1_A$.
Note that we often suppress notation for the inclusion map
and simply write  $d^{}_{{\BB_A}} =\pr^{}_{\BB_A}\, d^{}_{\B_A}$.

We require one small observation for
converting from unreduced to reduced resolutions:
\begin{equation}\label{ProjectD}
  \pr^{}_{\BB_A}\, d^{}_{\B_A} \, (1_{\B_A}-\inc^{}_{\BB_A}\pr^{}_{\BB_A}) \equiv 0
  \quad\text{ on } \B_A
  \, .
\end{equation}
This holds since $\Ima(1_{\B_A}-\inc^{}_{\BB_A}\pr^{}_{\BB_A})=\ker
\pr^{}_{\BB_A}$, which
is mapped to itself under $d_{\B_A}$ as this differential
(see \cref{eqn:bar-diff}) preserves
the span of pure tensors in $\B_A$ with a scalar for some inner tensor 
component due to cancellation of other pure tensors.
For example, 
$$
\begin{aligned}
  d^{}_{\B_A}& (1_A \ot a_1\ot 1_A\ot a_3\ot 1_A) 
 \\ & =
 a_1\ot 1_A\ot a_3\ot 1_A 
 - 1_A\ot a_1\ot a_3\ot 1_A 
 + 1_A\ot a_1\ot a_3\ot 1_A 
 -   1_A\ot a_1\ot 1_A \ot a_3 
  \\ & =
 a_1\ot 1_A\ot a_3\ot 1_A 
 -   1_A\ot a_1\ot 1_A \ot a_3 
\end{aligned}
$$
which projects to zero under $\pr^{}_{\BB_A}$. 

\subsection*{Iterated twisting}
We give notation to iterations extending the twisting map $\tau$
to the bar complex as in~\cite{GG, SW-twisted}:
Define
\begin{equation}\label{iterativetwisting}
  \tau^{}_{_{\B_R}}: S\ot \B_R\longrightarrow \B_R\ot S
\quad\text{ and }\quad
\tau^{}_{_{\B_S}}: \B_S\ot R\longrightarrow R\ot \B_S
\end{equation}
in each degree $n\geq 0$
as the obvious compositions (with $n+2$ factors)
$$
\begin{aligned}
\tau^{}_{_{\B_R}}
&=
(1_R\ot\cdots\ot 1_R\ot \tau) 
\ 
\cdots 
\ 
(1_R\ot \tau\ot 1_R\ot\cdots\ot 1_R)\ 
(\tau\ot 1_R\ot\cdots\ot 1_R) 
\, , 
\\
 \tau^{}_{_{\B_S}}
&=(\tau\ot 1_S\ot\cdots\ot 1_S) 
\  \cdots \ \,
(1_S\ot \cdots\ot 1_S\ot \tau\ot 1_S)
\ \, (1_S\ot\cdots\ot 1_S\ot \tau) 
\,  .
\end{aligned}
$$
We obtain versions for the reduced bar complexes 
$$
\tau_{_{\BB_R}}: S\ot \BB_R\longrightarrow \BB_R\ot S
\qquad\text{ and }\qquad
\tau_{_{\BB_S}}: \BB_S\ot R\longrightarrow R\ot \BB_S
$$
by composing
in the obvious way with projection and inclusion:
Set
\begin{equation}\label{iterativetwistingreduced}
\begin{aligned}
\tau_{_{\BB_R}}
&=(\pr^{}_{\BB_R}\ot 1_S)\ \tau_{_{\B_R}}(1_S\ot \inc^{}_{\BB_R})
\, 
\qquad\text{ and }\qquad
\\
\tau_{_{\BB_S}}
&=(1_R\ot \pr^{}_{\BB_S} )\, \tau_{_{\B_S}}\, (\inc^{}_{\BB_S}\ot 1_R)
\end{aligned} 
\end{equation}
for inclusion maps and projection maps
(see \cref{DifferentialReduced})
$$
\begin{aligned}
&\inc^{}_{\BB_R}: \BB_R\hookrightarrow \B_R
&&\text{ and }\quad
\pr^{}_{\BB_R}: \B_R \twoheadrightarrow \BB_R
\, ,
\\
&
\inc^{}_{\BB_S}:\, \BB_S\hookrightarrow \B_S\,
&&\text{ and }\quad
\pr^{}_{\BB_S}\,  : \B_S \twoheadrightarrow \BB_S
\, .
\end{aligned}
$$
Again, we often suppress notation for the inclusion maps to avoid clutter
and simply write
$\tau_{_{\BB_R}}
=
(\pr^{}_{\BB_R}\ot 1_S)\ \tau_{_{\B_R}}$
and
$
\tau_{_{\BB_S}}
=(1_R\ot \pr^{}_{\BB_S} )\, \tau_{_{\B_S}}
$.

 As  twisting maps
  commute with multiplication 
  (see \cref{eqn:taumm} and
  \cref{TwistingDefDiagram}),
  so does iterative twisting:
\begin{lemma}
   \label{twistingcommutesgeneralized}
Iterative twisting commutes with multiplication in the algebras $R$ and $S$:
For any $\ell, j \geq 0$, the following diagram of vector spaces is commutative
 \begin{small}
   \begin{equation*}
    \hspace{5ex}
\mbox{
\entrymodifiers={+!!<0pt,\fontdimen22\textfont2>}
\xymatrixcolsep{6ex}
\xymatrixrowsep{6ex}
\xymatrix{
S\ot R^{\ot \ell}\ot R\ot R \ot R^{\ot j} 
\ar[rrr]^{1_S\ot 1_R^{\ot \ell} \ot\, m_{_R}\ot 1_R^{\ot j}}
\ar[d]_{\tau_{_{\B_R}}  }
&  & &
S\ot R^{\ot \ell}\ot R\ot R^{\ot j}
\ar[d]^{\tau_{_{\B_R}} } 
& &
\\
R^{\ot \ell}\ot R\ot R \ot R^{\ot j}\ot S 
\ar[rrr]_{1_R^{\ot \ell} \ot \,m_{_R}\ot 1_R^{\ot j}\ot 1_S}
& & &
R^{\ot \ell}\ot R\ot R^{\ot j}\ot S
\, ,
& &
 \, 
}}
\end{equation*}
\end{small}%
and a similar diagram holds for $\tau_{_{\B_S}}$ with the roles of $S$ and $R$ reversed.
\end{lemma}

\subsection*{Compatibility of bar resolutions}
Iterated twisting provides compatibility maps
(see \cref{CompatibleResolutionsDef}).
A sketch of the proof can be found in \cite[Prop.~2.20(ii)]{SW-twisted};
we include here a more detailed statement and
complete argument for later use.
\begin{lemma}
  \label{BarCompatible}
Bar resolutions are compatible with the twisting map $\tau:S\ot R\rightarrow R \ot S$:
\begin{itemize}\setlength{\itemindent}{-5ex}
\item
  $\B_R$ and $\B_S$ via 
$\tau_{_{\B_R}}:\ S \ot \ \B_R \rightarrow \B_R \ot S$
and
$\tau_{_{\B_S}}:
\B_S\ot R \rightarrow R \ot \B_S$, respectively,
\item 
  $\, \BB_R\, $ 
  and  $\BB_S$ via 
$
\tau_{_{\BB_R}}: \ \, S \ot \ \, \BB _R
\rightarrow \, \BB_R   \ot S$
and
$\tau_{_{\BB_S}}:
\  \BB_S\ot R \rightarrow R\ot \, \BB_S
$, respectively.
\end{itemize}
\end{lemma}
\begin{proof}
\cref{twistingcommutesgeneralized}
implies that 
$\tau_{_{\B_R}}$ 
defines a bijective chain map that
commutes with the $R$-bimodule structure of $\B_R$
and multiplication in $S$
(see \cref{TwistingDefDiagram,CompatibleResolution} and \cref{eqn:bar-diff}).
To verify that
$\tau_{_{\BB_R}}$ is also
a chain map, we
use the fact that iterated twisting 
preserves $\ker\pr^{}_{\BB_R}$.
For $1$ the identity map with subscript indicating domain,
see \cref{DifferentialReduced,iterativetwistingreduced},
$$
\begin{aligned}
\tau_{_{\BB_R}}(1_S\ot d^{}_{\BB_R}) \ \ 
&= \ \
(\pr^{}_{\BB_R}\ot 1_S)\, \tau_{_{\B_R}}\, (1_S\ot \inc^{}_{\BB_R})
\big( 1_S\ot \pr^{}_{\BB_R} \, d^{}_{\B_R} \, \inc^{}_{\BB_R})
\\
&= \ \ 
(\pr^{}_{\BB_R}\ot 1_S)\, \tau_{_{\B_R}}\, 
\big( 1_S\ot (1_{\B_R}+ \inc^{}_{\BB_R}\, \pr^{}_{\BB_R}-1_{\B_R} ) \big) \, (1_S\ot 
d^{}_{\B_R} \, \inc^{}_{\BB_R})
\\ &=\ \ 
(\pr^{}_{\BB_R}\ot 1_S)\, \tau_{_{\B_R}}\, (1_S\ot d^{}_{\B_R}
\, \inc^{}_{\BB_R})
\\ &  \qquad +
(\pr^{}_{\BB_R}\ot 1_S)\, \tau_{_{\B_R}}\,\big( 1_S\ot
(\inc^{}_{\BB_R}\,\pr^{}_{\BB_R}-1_{\B_R} ) \big) \, (1_S\ot d^{}_{\B_R} \, \inc^{}_{\BB_R})
\, .
\end{aligned}
$$
The second summand is the zero map since
$$
\tau_{_{\B_R}}\ \big(1_S\ot \ima(\inc^{}_{\BB_R}\,\pr^{}_{\BB_R}-1_{\B_R} ) 
\big) 
=
\tau_{_{\B_R}}\big(1_S\ot \ker\pr^{}_{\BB_R} \big) 
\subset 
\ker(\pr^{}_{\BB_R})\ot k1_S 
\, $$ 
as $\tau(1_S\ot r)=r\ot 1_S$ 
for all $r$ in $R$.
Then by \cref{ProjectD}, and as $\tau_{_{\B_R}}$ is a chain map, 
$$
\begin{aligned}
\tau_{_{\BB_R}} (1_S\ot d^{}_{\BB_R}) \ \ 
 &=\ \ \
 (\pr^{}_{\BB_R}\ot 1_S)\, \tau_{_{\B_R}}\,  (1_S\ot d^{}_{\B_R}) \, (1_S\ot \inc^{}_{\BB_R}) 
\\ &=\ \ \ 
(\pr^{}_{\BB_R}\ot 1_S)\,  (d^{}_{\B_R}\ot 1_S)\, \tau_{_{\B_R}}
\, (1_S\ot \inc^{}_{\BB_R}) 
\\ &= \ \ \
(\pr^{}_{\BB_R}\,  d^{}_{\B_R}\, \inc^{}_{\BB_R}\,\pr^{}_{\BB_R}\ot 1_S)\, \tau_{_{\B_R}}
\, (1_S\ot \inc^{}_{\BB_R}) 
\\ & \qquad +
\big(  \pr^{}_{\BB_R}\, d^{}_{\B_R}\, (1_{\B_R}-\inc^{}_{\BB_R}\,\pr^{}_{\BB_R})\ot 1_S\big)\,
\tau_{_{\B_R}}
\, (1_S\ot \inc^{}_{\BB_R}) 
\\ &=\ \ \
(\pr^{}_{\BB_R}\, d^{}_{\B_R}\, \inc^{}_{\BB_R}\ot 1_S)\, (\pr^{}_{\BB_R}\ot 1_S)\, 
\tau_{_{\B_R}}
\, (1_S\ot \inc^{}_{\BB_R}) 
\\ &=\ \ \
(d^{}_{\BB_R}\ot 1_S)\, \tau_{_{\BB_R}}
\end{aligned}
$$
and hence $\tau_{_{\BB_R}}$ is a chain map.
A similar argument may be used to confirm that $\tau_{_{\BB_R}}$ commutes
with the $R$-bimodule
structure of $\BB_R$ and the multiplication in $S$ (see
\cref{CompatibleResolution}).

Lastly, we note that  $\tau_{_{\BB_R}}$ is
bijective. Indeed, 
$\tau_{_{\BB_R}}$ may be written as a composition 
of bijective maps 
$(\tau\ot 1_R\ot \cdots\ot 1_R)$ and $(1_R\ot\cdots\ot 1_R\ot \tau)$
and maps of the form
$$1_R\ot \cdots\ot 1_R\ot (\pr^{}_{R}\ot 1_S)\, \tau\ot 1_R\ot
\cdots\ot 1_R\, 
$$
as $\tau$ is just the swap (interchange) map on $S\ot k1_R$.
The map 
$(\pr^{}_{\bar R}\ot 1_S)\, \tau : S\ot \bar R \rightarrow \bar R \ot S$ is injective:  
If  $s\ot \bar r$ lies in its kernel
for $s$ in $S$ and $\bar r$ in $\bar R$, 
 then
$\tau (s\ot \bar r)$ lies in $k1_R\ot S$
and hence $s\ot \bar r$ lies in $\tau^{-1}(k1_R\ot S)=
S\ot k1_R$, so must be zero, and similarly for sums of such elements. 
This map 
is also surjective
as it has right inverse $(1_S\ot \pr^{}_{\bar R})\,\tau^{-1}$ on
$\bar R\ot S$.

Hence 
$\BB_R$ 
is compatible with $\tau$.
An analogous argument applies to 
$\BB_S$.
\end{proof}

\subsection*{Twisting bar resolutions}
The twisted product resolutions
$\B_R \ot_\tau \, \B_S$ 
and 
$\BB_R \ot_\tau \, \BB_S$ 
have respective $n$-th degree terms
$$
\begin{aligned}
(\B_R \ot_{\tau} \, \B_S)_n 
&\ =\ \bigoplus_{\ell=0}^{n}\ 
 R\ot R^{\,\ot (n-\ell)}\ot R \ot 
 S\ot {S}^{\, \ell}\ot S\
 \quad\text{ and }
 \\
 (\BB_R \ot_{\tau} \, \BB_S)_n 
&\ =\  \bigoplus_{\ell=0}^{n}\ 
 R\ot \bar{R}^{\,\ot (n-\ell)}\ot R \ot 
 S\ot \bar{S}^{\, \ell}\ot S
 \, .
\end{aligned}
 $$
We compare with the bar and reduced bar resolutions of $\ttp$:
$$
\begin{aligned}
(\B_{\ttp})_n 
&=\ttpp \ot ({\ttp})^{\ot n} \ot \ttpp
\\
(\BB_{\ttp})_n 
&=\ttpp \ot (\overline{\ttp})^{\ot n} \ot \ttpp
\, .
\end{aligned}
$$

\subsection*{Embedding twisted product resolutions}
We identify as $k$-vector spaces (see Eq.~(\ref{eqn:bar-id}))
\begin{equation}\label{ReducedTwistedTensorProduct}
  \overline{R\ott S}=(R\ott S)/ k\cdot (1_{R\ott S})
\quad\text{ with }\quad
(\bar{R}\ott \bar{S})\oplus (\bar{R}\ott k1_S)\oplus (k1_R\ott
\bar{S})
\, .
\end{equation}
Note here that $1_{R\ott S}=1_R\ot 1_S$.
This gives rise to an embedding 
$(\BB_{R\ott S})_n \subset (\B_{R\ott S})_n$ for each $n$.
Thus a choice of
sections of the vector space projection maps 
$R\twoheadrightarrow \bar{R}=R / k1_R$ and $S\twoheadrightarrow 
\bar{S}=S/ k 1_S$ give embeddings of bimodules in each degree 
\begin{equation}
  \label{embeddings}
  (\BB_R)_n \subset (\B_R)_n , 
\quad 
  (\BB_S)_n \subset (\B_S)_n, 
\quad 
  (\BB_{R\ott S})_n \subset (\B_{R\ott S})_n
 \, . 
\end{equation}

\section{Twisted Alexander-Whitney and Eilenberg-Zilber maps}
\label{sec:AW-EZ}

For $k$-algebras $R$ and $S$ and a twisting map
$\tau:S\ot R\rightarrow R\ot S$,
we describe twisted Alexander-Whitney and Eilenberg-Zilber
maps on the reduced bar resolutions,
$$
\begin{aligned}
  \AW^{\tau} :  \BB_{\ttps} \ \longrightarrow   \ 
  \BB_R\ot_{\tau} \, \BB_S 
  \,  \quad\text{ and }\quad
 \ \, \EZ^{\tau} :   
  \BB_R\ot_{\tau} \, \BB_S 
    \ \longrightarrow  \ 
          \BB_{\ttps}\,  , 
\end{aligned} 
$$ 
by first defining versions on
the unreduced bar resolutions,
$$
\begin{aligned}
 \ \  \AWtB :  \B_{\ttps} \ \longrightarrow   \ 
  \B_R\ot_{\tau} \, \B_S 
  \,  \quad\text{ and }\quad
  \EZtB :   
  \B_R\ot_{\tau} \, \B_S 
    \ \longrightarrow \ 
          \B_{\ttps}\,  .
\end{aligned} 
$$
These are given as iterations of the twisting map $\tau$
combined with
perfect shuffles/unshuffles on tensor components
and versions of the traditional front/back face maps
and shuffle product maps
(see \cite{May}, for example):
Compare with maps in~\cite{GG};
we include details here for completeness and later use.

\subsection*{Intermediate complex}
We use an intermediate $(R\ott S)$-bimodule
resolution $Y\sd$ of $R\ott S$
with $n$-th term
$Y_n=R^{\ot n+2}\ot S^{\ot n+2}$:
Consider the complex
\begin{equation}
  \label{IntermediateComplex}
\cdots\xrightarrow{ \ \ \Ydiff \ \ }
Y_1=R^{\ot 3}\ot S^{\ot 3}
\xrightarrow{ \ \  \Ydiff \ \ } 
Y_0=R\ot R \ot S \ot S
\xrightarrow{m_{_R}\ot \, m_{_S}}
R\ott S \longrightarrow 0\,
\, 
\end{equation}
with
differential 
$
\Ydiff_n: Y_n
 \rightarrow 
Y_{n-1}
$  
defined by
$$
\Ydiff_n=\sum_{\ell=0}^n\ (-1)^\ell \
1_R^{\otimes \ell}\ot m_{_R}\ot 1_R^{n-\ell}
\ot
1_S^{\otimes \ell}\ot m_{_S}\ot 1_S^{n-\ell}\, .
$$
We imbue each term in (\ref{IntermediateComplex}) 
with the structure of an
$(R\ott S)$-bimodule
in a straightforward way
using the iterative twisting maps of (\ref{iterativetwisting})
and Diagram~(\ref{bimodstructure}):
Define a left-right action
of $R\ott S$ on each $Y_{n}$,
$$
(R\ott S)\ot Y_n\ot (R\ott S) 
\longrightarrow 
Y_n\, ,
$$
given as the composition
\begin{small}
$$
(m_{_R}\ot 1_R^{\ot n}\ot m_{_R}\ot 
m_{_S}\ot 1_S^{\ot n}\ot m_{_S}) 
\
(1_R\ot 1_R^{\ot (n+2)}\ot \tau\ot 1_S^{\ot (n+2)}\ot 1_S)
\ 
(1_R\ot \tau^{}_{\B_R} \ot \tau^{}_{\B_S}\ot 1_S) 
. 
$$
\end{small}%
In other words, $R$ acts on the left by left multiplication
in the leftmost tensor component, and $R$ acts on the right
by right multiplication on the rightmost tensor component from $R$ 
after iterated twisting through the components from $S$, 
and similarly for the action of $S$ (reversing left-right roles).
Then as $\tau$ is bijective, each term $Y_n$
of the complex is a free $(R\ott S)$-bimodule
with free basis given by a $k$-vector space basis of 
$$
1_R\ot R^{\ot n} \ot 1_R\ot 1_S\ot S^{\ot n} \ot 1_S
\, .
$$
One may check that the differential $\Ydiff\sd$
is an $(R\ott S)$-bimodule map using 
\cref{twistingcommutesgeneralized}
and hence $Y\sd$ is indeed an $(R\ott S)$-bimodule resolution.
 
\subsection*{Action of the symmetric group}

We take a standard action of the symmetric group $\mathfrak{S}_n$ on
any $n$-fold tensor product of $k$-vector spaces $V_i$ for $i=1,\ldots,n$:
$$
\begin{aligned}
 \s:
V_1\ot  \cdots\ot  V_n
&\ \longrightarrow\
V_{\s^{-1}(1)}\ot \cdots\ot V_{\s^{-1}(n)}
\ , 
\\
v_1\ot \cdots \ot v_n 
&\ \overset{\s}{\longmapsto}\
v_{\s^{-1}(1)}\ot\cdots\ot v_{\s^{-1}(n)}
\quad\text{ for } \s \text{ in } \mathfrak{S}_n,\ v_i \text{ in } V_i \, .
\end{aligned}
$$
For $n=0$, we take $\mathfrak{S}_0=\{1\}$ acting trivially on $k=V^{\ot 0}$. 

\subsection*{Shuffles}
Recall that a permutation $\s$ in $\mathfrak{S}_n$ is an
$(\ell, n-\ell)$-{\em shuffle}
when
$$
\s(1) < \cdots< \s(\ell)
\quad\text{ and }\quad
\s(\ell+1) < \cdots< \s(n)\, .
$$
We denote the set of $(\ell, n-\ell)$-shuffles by $\mathfrak{S}_{\ell, n-\ell}\subset \mathfrak{S}_n$.
The 
{\em perfect (out) unshuffle}  in $\mathfrak{S}_{2n}$ 
acts as the vector space map
\begin{equation}
  \label{unshuffle}
  \begin{aligned}
    (R\ot S)^{\ot n}
    &\longrightarrow R^{\ot n}\ot S^{\ot n}\, ,
    \, \\
  (r_1\ot s_1) \ot\cdots\ot (r_n\ot s_n)
  &\longmapsto
  (r_1\ot\cdots\ot r_n) \ot (s_ 1\ot\cdots\ot s_n)\, .
  \end{aligned}
  \end{equation}
Its inverse is the {\em perfect (out) shuffle}, i.e.,
the $(n,n)$-shuffle in $\mathfrak{S}_{2n}$ defining the map 
$$
\begin{aligned}
  R^{\ot n}\ot S^{\ot n} &\longrightarrow (R\ot S)^{\ot n}
  ,\\
(r_1\ot \cdots\ot r_{n}) \ot (s_1\ot \cdots\ot s_{n})
&\longmapsto
(r_1\ot s_1)\ot\cdots\ot (r_{n}\ot s_{n})\, .
\end{aligned}
$$

\subsection*{Twisted perfect shuffle and unshuffle}
We construct twisted analogues.
Define the {\em twisted perfect unshuffle
map} $\varrho: \B_{\ttp} \rightarrow Y\sd$
in each degree $n\geq 0$ by
\begin{equation}
\label{twistedperfectunshufflemap}
\varrho_n : 
(R\ot S)^{\ot (n+2)}\longrightarrow R^{\ot (n+2)}\ot S^{\ot (n+2)}
\hspace{10ex}
\text{ (twisted perfect unshuffle)} 
\end{equation}
given as the composition 
$$
\varrho_n
=
(1_R^{\ot n+1} \ot \tau \ot 1_S^{\ot (n+1)})
\
(1_R^{\ot n} \ot \tau^{\ot 2} \ot 1_S^{\ot n})
\
\cdots
\ 
(1_R\ot \tau^{\ot (n+1)} \ot 1_S)
\, ,
$$
and define the {\em twisted perfect shuffle map}
$\varrho^{-1}: Y\sd\rightarrow \B_{\ttp}$
as its inverse:
$$
\varrho^{-1}_n: 
R^{\ot (n+2)}\ot S^{\ot (n+2)}
\longrightarrow
(R\ot S)^{\ot (n+2)}
\hspace{10ex}
\text{ (twisted perfect shuffle)}
\, .
$$
An application of \cref{twistingcommutesgeneralized} 
verifies that $\varrho$ and $\varrho^{-1}$
are $\ttpp$-bimodule maps.

\vspace{2ex}

\begin{example}\label{GroupUnshuffle}
  {\em
    Consider the group ring $kG$ of
    a finite group $G\subset \GL(V)$ acting on a vector space
    $V\cong k^n$ and extend
    to an action by automorphisms on $k[V]=k[x_1, \ldots, x_n]$, the polynomial ring
    on a basis of $V^*$ (or alternatively, $V$), denoted $f\mapsto \
    ^{g}\!f$.
    The semidirect product algebra $k[V]\rtimes G$ is called the
    {\em skew group algebra}
    or {\em smash product algebra} 
    and 
    is the twisted tensor product
      $$k[V]\ott kG=k[V]\rtimes G=k[V]\# kG$$
given by twisting map
              \begin{equation} 
                \tau: kG\ot k[V]\longrightarrow k[V]\ot kG
                \qquad\text{ defined
                by }
              g\ot f \mapsto\ {}^g\! f\ot g\, .
            \end{equation}
    For $n=2$,
    the twisted perfect unshuffle 
    $\varrho_2:(k[V] \ot kG)^{\ot 4} \mapsto k[V]^{\ot 4} \ot (kG)^{\ot 4}$
    is given by 
    $$
    \begin{aligned}
       f_1\ot g_1\,\ot & \, f_2\ot g_2\ot f_3\ot g_3\ot f_4\ot g_4 
      \\
      &\longmapsto 
      f_1\ot  {}^{g_1}\! f_2\ot {}^{g_1g_2}\! f_3\ot {}^{g_1g_2g_3}\! f_4\ot 
      g_1\ot    g_2\ot g_3\ot g_4
      \, ,
\end{aligned}
  $$
  whereas
    the twisted perfect shuffle 
    $\varrho^{-1}_2: k[V]^{\ot 4} \ot (kG)^{\ot 4} \mapsto   (k[V]\ot kG)^{\ot 4} 
    $
    is given by 
    $$
    \begin{aligned}
      f_1\ot f_2\, \ot & \, f_3\ot f_4\ot g_1\ot g_2\ot g_3\ot g_4 
      \\
      &\longmapsto 
      f_1\ot g_1\ot\, {}^{g_1^{-1}}\! f_2\ot g_2\ot
      \, {}^{(g_1g_2)^{-1}}\! f_3\ot g_3 \ot\, {}^{(g_1g_2g_3)^{-1}}\!
      f_4\ot g_4
      \, .
    \end{aligned}
  $$
}\end{example}

\vspace{2ex}

\subsection*{Front/back face map}
Let
$\phi: Y\sd\rightarrow \B_R \ott \B_S$
be the {\em front/back face map}
defined in each degree $n\geq 0$ by
\begin{equation}
  \label{frontbackmap}
\phi_n:
R^{\ot (n+2)} \ot S^{\ot (n+2)} 
\longrightarrow
(\B_R \ott \B_S)_n
\, , \qquad
\phi_n = \sum_{\ell=0}^n\ (-1)^{\ell(n-\ell)}\  \phi_{n-\ell,\ell}
\, ,
\end{equation}
for
 $ \phi_{n-\ell, \ell}: R^{\otimes (n+2)}\ot S^{\otimes (n+2)}
\rightarrow 
(\B_R)_{n-\ell} 
\ot 
(\B_S)_{\ell}
$
the $\ttpp$-bimodule map
given by

\vspace{-.5ex}

\begin{small}
$$
\begin{aligned}
1_R\ot r_1\ot \cdots\ot r_n\ot 1_R
& \ot 1_S\ot s_1\ot\cdots\ot s_n \ot 1_S
\\
&\longmapsto\
(r_1\cdots r_{\ell})\ot r_{\ell+1}\ot \cdots\ot r_n\ot 1_R
\ot 
1_S\ot s_1\ot\cdots\ot s_{\ell}\ot (s_{\ell+1}\cdots s_n)
\, .
\end{aligned}
$$
\end{small}%


\subsection*{Shuffle map}
Let $\theta:\B_R\ott \B_S \rightarrow Y\sd$
be the {\em shuffle map}
defined
in each degree $n\geq 0$ by
\begin{equation}
  \label{shufflingsummap}
  \theta_n: (\B_R \ott \B_S)_n\longrightarrow
  Y_n=R^{\ot (n+2)} \ot S^{\ot (n+2)}
\, , \qquad
 \theta_n = \bigoplus_{\ell =0}^n\ (-1)^{\ell(n-\ell)}\ \theta_{\ell, n-\ell}
 \, ,
\end{equation}
for
$\theta_{\ell, n-\ell}: 
(\B_R)_{\ell} \ot ( \B_S)_{n-\ell}
\rightarrow R^{\ot (n+2)}\ot S^{\ot (n+2)}
$
 the $\ttpp$-bimodule map given by

\vspace{-.5ex}

\begin{small}
$$
\begin{aligned} 
1_R\ot r_1 \ot\cdots  
 \ot \, &\, r_{\ell} \ot 1_R
 \ot 1_S\ot s_{\ell+1}\ot\cdots\ot s_{n} \ot 1_S\\
&\longmapsto 
\sum_{\s \in \mathfrak{S}_{\ell, n-\ell}} \ 
\sgn(\s)\ 
1_R\ot \s \big( r_1\ot\cdots\ot r_\ell \ot 
\underbrace{1_R\ot \cdots\ot 1_R\rule[-1ex]{0ex}{2ex}}_{ (n-\ell)\text{-times} \rule{0ex}{1ex}}  \big) \ot 1_R\\
&\hspace{22ex}
 \ot\ 
1_S \ot \s \big(
\underbrace{1_S\ot \cdots\ot 1_S\rule[-1ex]{0ex}{2ex}}_{ \ell\text{-times} \rule{0ex}{1ex}} 
 \ot \, s_{\ell+1}\ot\cdots\ot s_{n} \big) \ot 1_S\, .
\end{aligned}
$$
\end{small}

\vspace{2ex}

\begin{example}{\em
    For $n=5$, $\ell=2$,
the summand of
$$\theta_{2,3}:(R\ot R\ot R \ot R)
\ot (S\ot S \ot S \ot S \ot S)
\longrightarrow
R^{\ot 7}\ot S^{\ot 7}
$$
corresponding
to the $(2,3)$-shuffle 
$\s=(1\ 3)(2\ 5\ 4)$ in $\mathfrak{S}_{2,3}\subset\mathfrak{S}_5$
applied to
the element $1_R\ot r_1\ot r_2\ot 1_R\ot 1_S \ot s_1\ot s_2\ot s_3 \ot 1_S$
is
$$
\begin{aligned}
1_R\, \ot & \, \s\big(r_1\ot r_2 \ot 1_R\ot 1_R \ot 1_R\big)\ot 1_R
\ot
1_S\ot \s\big(1_S\ot 1_S \ot s_3\ot s_4\ot s_5\big)\ot 1_S\\
&=
1_R\ot (1_R \ot 1_R \ot r_1\ot 1_R \ot r_2)\ot 1_R
\ot
1_S\ot (s_3\ot s_4\ot 1_S \ot s_5\ot 1_S) \ot 1_S
\, .
\end{aligned}
$$
}
\end{example}

\vspace{2ex}

\subsection*{Construction of twisted AW and
  EZ maps}

We now
define the twisted Alexander-Whitney and Eilenberg-Zilber maps
using 
the {twisted perfect unshuffle} map $\varrho$ of~(\ref{twistedperfectunshufflemap}), 
the {front/back face map}  $\phi$ of \cref{frontbackmap}, 
and the shuffle map $\theta$  of \cref{shufflingsummap}:
Define $\ttpp$-bimodule homomorphisms
$$
\AWtB: \B_{\ttp}\longrightarrow \B_R\ott\B_S
\quad\text{ and }\quad
\EZtB: \B_R\ott\B_S\longrightarrow \B_{\ttp}
$$
in each degree $n\geq 0$ as the compositions
$$
(\AWtB)_n =  \phi_n \ \varrho_n
\qquad\text{ and }\qquad
(\EZtB)_n \ =  
\varrho_n^{-1}\ \theta_n     
\, .
$$
We use the intermediate complex $Y$
to take advantage of traditional results
in the untwisted case
(see \cite{May}, for example):
\begin{equation*}
\begin{tikzcd}[row sep=huge, column sep=small, ampersand replacement=\&]
  \hphantom{A}
  \arrow[swap, bend right=50, dashed]{dd}{\AWtB}
   \hphantom{A}
   \&[-3em]
   (\B_{R\ot_{\tau} S})_{n}
   \arrow[swap]{d}{\varrho_{n}} 
   \&[-3em]
   \hphantom{A}
   \\
    \& Y_n=R^{\ot(n+2)}\ot S^{\ot(n+2)}
  \arrow[swap,shift right=2]{u}{\varrho_{n}^{-1}
    \hspace{12ex} \substack{\quad \ \text{twisted perfect shuffle (up)}    \\
      \hphantom{x,}\quad\quad\ \, \text{twisted perfect unshuffle (down)}}}
  \arrow[swap]{d}{\phi_{n}}
  \&
   \\
  \hphantom{A}\rule{0ex}{2ex}  
   \&  (\B_{R} \ott \B_S )_{n} 
  \arrow[swap,shift right=2]{u}{\theta_{n}
    \hspace{11.5ex} \substack{  \text{shuffle map (up)}\\
      \qquad\quad\ \ \hphantom{.} \text{front/back face map (down)}}}
   \&
   \  \arrow[bend right=50, swap, dashed]{uu}{\EZtB}
\end{tikzcd}
\end{equation*}

\vspace{2ex}

\begin{example}\label{GroupExampleAWEZmaps}
    {\em
For a finite group $G\subset \GL(V)$ acting on $V=k^n$
with $k[V]=k[x_1,\ldots, x_n]$
as in \cref{GroupUnshuffle},
we  suppress tensor symbols and write $fg$ for $f\ot g$ in $k[V]\ott
kG=k[V]\# G$
for $\tau: g\ot f\mapsto  \, ^g\! f\ot g$ with $g$ in $G$ and $f$ in $k[V]$.
Then
$$\AW^{\tau}_{\B}: \B_{k[V]\# kG}\longrightarrow \B_{k[V]}\ot_{\tau}\B_{kG}$$
is given by
\begin{small}
$$
\begin{aligned}
&(\AW^{\tau}_{\B})_n(1\ot f_1g_1\ot\cdots\ot f_ng_n\ot 1) \\
& = \sum_{\ell=0}^n (-1)^{\ell (n-\ell)}
\big(f_1     ({}^{g_1}\! f_2)  ( {}^{g_1g_2}\! f_3)\cdots ({}^{g_1g_2\cdots g_{\ell
         -1}}\! f_{\ell}) \big)
   \ot ({}^{g_1g_2\cdots g_{\ell}}\! f_{\ell +1}) \ot\cdots
   \ot ( {}^{g_1g_2\cdots g_{n-1}}\! f_n) \ot 1_{k[V]} \\
   &  \hspace{45ex} \ot 1_G
   \ot g_1\ot \cdots\ot g_{\ell}\ot (g_{\ell+1}\cdots g_n)
   \, ,
\end{aligned}
$$
\end{small}
for $f_i$ in ${k[V]}$ and $g_i$ in $G$.
This formula defines a chain map for the reduced bar resolutions,
$\AW^{\tau} : \BB_{{k[V]}\# kG} \rightarrow \BB_{k[V]} \ott \BB_{kG}$,
by interpreting any tensor component except the first and last
(ignoring the middle term $1_{k[V]}\ot 1_G$)  in $k$
as zero,
see \cref{AW-EZforreducedbar} and its proof.

Rather than giving a closed formula for $\EZtB$ in this setting,
we just give a quick example 
of
$(\EZtB)_3$ mapping $(\B_{{k[V]}})_1\ot 
(\B_{kG})_2$ to $(\B_{{k[V]}\# kG})_3$:  For $A={k[V]}\# kG$, $f$ in
$k[V]$,
$g_i$ in $G$, 
and extra factors $1_{k[V]}$ and $1_G$ suppressed,
\begin{small}
$$
\begin{aligned}
  & \EZtB(1_{k[V]} \ot f\ot 1_{k[V]}\ot 1_G\ot g_2\ot g_3\ot 1_G)
\\ 
 &\ \ = 1_{A}\ot f\ot g_2\ot g_3\ot  1_{A}
  -1_A\ot 
g_2\ot \, ^{g_2^{-1}}\! f \ot g_3\ot 1_A 
+1_A\ot  g_2\ot  g_3\ot \ ^{(g_2 g_3)^{-1} } \! f \ot 1_A 
\\
\end{aligned}
$$
\end{small}
as the $(2,1)$-shuffles in $\mathfrak{S}_3$
are $(1)$, $(2\ 3)$, and $(1\ 2\ 3)$.
(Also see \cref{LieAlgebraExample}.)

Note that we may instead identify $k[V]\# kG$ with $kG\ott k[V]$ for the (inverse) twisting map
$
\tau: k[V]\ot kG \rightarrow kG \ot k[V]$
defined by
$f\ot g\mapsto g\ot \, ^{g^{-1}} \! f$
       and write
     $fg$ for $(1\ot f) (g\ot 1)$ in $kG\ott k[V]$.
     We then recover the twisted Alexander-Whitney
     and Eilenberg-Zilber maps
     for group actions
 constructed in~\cite{SW-grouptwistedAWEZ},
       and we correct here a typographical error
       on the last tensor component in the formula
       for
       $\AWtB: \B_{k[V]\# kG}\rightarrow \B_{kG}\ot \B_{k[V]}$:
       \begin{small}
         $$
\begin{aligned}
  &  (\AWtB)_n (1\ot   f_1g_1 \ot\cdots\ot f_ng_n \ot 1 ) \\
  &  \hspace{1em}  
  =  
  \sum_{\ell=0}^n (-1)^{\ell(n-\ell)} (g_1\cdots
   g_\ell) \ot g_{\ell+1}\ot\cdots \ot g_n\ot 1_G
     \\
&\hspace{3em}
  \ot 1_{k[V]}\ot{}^{(g_1\cdots g_n)^{-1}}\! f_1\ot\cdots\ot
  {}^{(g_{\ell}\cdots g_n)^{-1}}\! f_\ell
  \ot  ({}^{(g_{\ell+1}\cdots g_n)^{-1}}\! f_{\ell+1})
( {}^{(g_{\ell+2}\cdots g_n)^{-1}}\! f_{\ell+2})
\cdots
( {}^{g_n^{-1}}\! f_{n})
\end{aligned}
$$
\end{small}%
for $g_i$ in $G$ and $f_i$ in $k[V]$.
Again, this formula gives a chain map for the reduced bar resolutions
$\AW^{\tau} =\AW^G :
\BB_{S\# k[V]}\rightarrow \BB_{kG}\ott\, \BB_{k[V]}$ 
as in~\cite{SW-grouptwistedAWEZ} by interpreting any tensor component in
$k$ as zero
 (ignoring the middle term $1_G\ot 1_{k[V]}$).
}
\end{example}

\vspace{2ex}

\section{Twisted Alexander-Whitney and Eilenberg-Zilber 
maps as chain maps}
\label{sec:chainmaps}
We show in this section that the
twisted Alexander-Whitney and Eilenberg-Zilber maps 
constructed in \cref{sec:AW-EZ} 
for the (unreduced) bar resolutions
are chain maps.
Again, fix algebras $R$ and $S$ and a twisting map
$\tau:S\ot R\rightarrow R\ot S$.
 
\begin{thm}\label{botharechainmaps}
 The twisted Alexander-Whitney map $\AW_{\B}^\tau$
 and the Eilenberg-Zilber map $\EZ_{\B}^\tau$
 for (unreduced) bar resolutions
are both chain maps of $\ttpp$-bimodule complexes.
\end{thm}

\begin{proof}
We write $1$ for the identity map on $R$ or $S$.
As $\tau$ is a twisting map (see \cref{TwistingDefDiagram}),
$\AWtB$ and $\EZtB$ commute with the differentials in degree $0$.
For degrees $n> 0$, we use the intermediary complex
$Y\sd$ of \cref{sec:AW-EZ} and
argue that the diagram
\begin{small} 
\begin{equation*}
\begin{tikzcd}[row sep=5ex, column sep=5ex, ampersand replacement=\&]
  \hphantom{A}
  \arrow[swap, bend right=65, dashed]{dd}{\AWtB}
   \hphantom{A}
   \&[-3em]
   (\B_{R\ott S})_{n}
   \arrow[swap]{d}{\varrho_{n}}
   \arrow{r}{d}
   \& (\B_{R\ott S})_{n-1}
   \arrow[swap]{d}{\varrho_{n-1}}
  \\ 
  \& Y_n=R^{\ot(n+2)}\ot S^{\ot(n+2)}
  \arrow[swap,shift right=2]{u}{\varrho_{n}^{-1} }
  \arrow{r}{\Ydiff}
  \arrow[swap]{d}{\phi_{n}}
  \& 
  R^{\ot(n+1)}\ot S^{\ot(n+1)}=Y_{n-1}
  \arrow[swap,shift right=2]{u}{\varrho_{n-1}^{-1} 
    \hspace{1ex} \substack{\, \text{twisted perfect shuffle (up)}\\
      \quad\ \ \  \text{twisted perfect unshuffle (down)}}}
   \arrow[swap]{d}{\phi_{n-1}}
   \\
  \hphantom{AA}\rule{0ex}{2ex}  
  \  \arrow[bend left=45, swap, dashed]{uu}{\EZtB}
  \&  (\B_{R} \ott \B_S )_{n} 
  \arrow[swap,shift right=2]{u}{\theta_{n}}
  \arrow{r}{\del}
   \&  (\B_{R} \ott \B_S )_{n-1} 
  \arrow[swap,shift right=2]{u}{\theta_{n-1}
    \hspace{-5.5ex} \substack{\quad  \text{shuffle map (up)}\\
      \text{\qquad \qquad \qquad\,  front/back face map (down)}}}
\end{tikzcd}
\end{equation*}
\end{small}
with only downward arrows
is commutative 
and that the same diagram with only upward arrows is commutative.
The lower square 
is commutative by the traditional
method of acyclic models \cite{EilenbergMacLane},
see~\cite{May}.
For the upper square, 
we restrict the maps $\Ydiff_n=\sum_{\ell=0}^n \Ydiff_n^{\ell}$ 
and $d_n=\sum_{\ell=0}^n d_n^{\ell}$
to their $\ell$-th summands
and argue that
$\varrho_{n-1}\, d_{n}^\ell = \Ydiff_{n}^\ell \, \varrho_n$
  for $\ell=0,\ldots, n$,
for 
$$
\begin{aligned}
d_n^\ell &=
1_{R\ot S}^{\ot \ell}\ot
((m_{_R}\ot m_{_S}) (1_R\ot\tau\ot 1_S) ) \ot 1_{R\ot S}^{\ot (n-\ell
  )}
\quad\text{ and }
\\
\Ydiff_n^\ell &=  1_R^{\ot \ell} \ot m_{_R} \ot 1_R^{\ot (n-\ell )}
  \ot
  1_S^{\ot \ell} \ot m_{_S} \ot 1_S^{\ot (n-\ell )}
  \, .
  \end{aligned}
  $$
One may verify the case $n=1$, 
$\ell=0,1$ directly using
\cref{twistingcommutesgeneralized}
(see \cref{TwistingDefDiagram}).
Now assume inductively that
$\varrho_{n-1}\, d_{n}^\ell = \Ydiff_{n}^\ell \, \varrho_n$
for fixed $n\geq 1$ and $0\leq \ell \leq n$:
\begin{small} 
\begin{equation}
\label{inductivestep}
          \hspace{5ex}
\mbox{
\entrymodifiers={+!!<0pt,\fontdimen22\textfont2>}
\xymatrixcolsep{6ex}
\xymatrixrowsep{6ex}
\xymatrix{
  (R\ot S)^{\ot (n+2)}
  \ar[rrr]^{d_n^\ell}
\ar[d]_{\varrho_n}
&  & &
R^{\ot (n+1)} \ot S^{\ot (n+1)} 
\ar[d]^{\varrho_{n-1}}
& &
\\
R^{\ot (n+2)}\ot S^{\ot (n+2)}
\ar[rrr]^{\Ydiff_n^{\ell}}
& & &
R^{\ot (n+1)}\ot S^{\ot (n+1)}
\, .
& &
}}
\end{equation}
\end{small}%
\hspace{-1ex}
We tensor on the right with one extra copy
of $R\ot S$
and use  \cref{twistingcommutesgeneralized}
to obtain a commutative diagram
verifying the case $n+1$ and $\ell<n+1$.
We similarly verify the case $\ell=n+1$ by
tensoring Diagram~(\ref{inductivestep})
with one copy of
$R\ot S$ on the left instead.
Hence
$\varrho\, d = \Ydiff \, \varrho$.
Since $\varrho$ is invertible,
$d\, \varrho^{-1} = \varrho^{-1}\, \del$
and thus
$\AWtB$ 
and $\EZtB$ 
are chain maps.
\end{proof}

\section{
 Twisted AW and EZ
maps split the reduced bar resolution}
\label{sec:chainmapsonreduced}

We now use the 
twisted Alexander-Whitney map $\AWtB$
and
Eilenberg-Zilber map $\EZtB$
for unreduced bar resolutions constructed in \cref{sec:AW-EZ}
to define
chain maps $\AWt$ and $\EZt$ for reduced bar resolutions.
We show the composition $\AWt\, \EZt$
is the identity in this setting, providing
a splitting of the reduced
bar resolution $\BB_{R\ott S}$ for a twisted tensor product
$R\ott S$
with a copy of the resolution $\BB_R\ott \BB_S$
as a direct summand.
Compare with \cite[Theorem 1.5]{GG}.

\begin{thm}\label{AW-EZforreducedbar}
  For any $k$-algebras $R$ and $S$ and twisting map
  $\tau:S\ot R\rightarrow R\ot S$,
there exist
   $\ttpp$-bimodule
  chain maps for reduced bar complexes
\begin{equation*} 
  \begin{aligned}
\AWt:\ \ \BB_{\ttps}
\longrightarrow \BB_R\ot_{\tau} \, \BB_S \, 
\quad\text{ and }\quad 
\EZt:\ \ 
\BB_R\ot_{\tau} \, \BB_S \longrightarrow\BB_{\ttps}
\ 
\end{aligned}
\end{equation*}
satisfying $\BAW^\tau \, \BEZ^\tau=1$,
 the identity map.
\end{thm}
\begin{proof}
  Recall that we fix 
  arbitrary choices of sections $\bar R\hookrightarrow R$ and $\bar S\hookrightarrow S$ 
of the vector space projections 
  $R\twoheadrightarrow \bar{R}=R / k1_R$ and 
$S\twoheadrightarrow \bar{S}=S / k1_S$ 
to embed $\bar{R}$ in $R$ and $\bar{S}$ in $S$, respectively, and to 
embed each bimodule component of the 
reduced into the unreduced bar complex
(see \cref{embeddings}):  
$$(\BB_R)_n \subset (\B_R)_n , 
\quad 
(\BB_S)_n \subset (\B_S)_n , 
\quad
(\BB_{R\ott S})_n \subset (\B_{R\ott S})_n \
\, .
$$
We extend to inclusion maps 
$$\inc_1 : (\BB_{R\ott S})_n \hooklongrightarrow (\B_{R\ott S})_n
\qquad\text{ and }\qquad 
\inc_2: 
(\BB_R\ott\BB_S)_n \hooklongrightarrow (\B_R\ott\B_S)_n 
$$ 
and  projection maps (induced componentwise)
$$
\pr_1
:\B_{R\ott S}\longtwoheadrightarrow \BB_{R\ott S}
\qquad\text{ and }\qquad 
\pr_2
:\B_R\ott\B_S\longtwoheadrightarrow \BB_R\ott\BB_S
\, .
$$
Note that $\pr_1$ 
applies the projection
$R\ot S \twoheadrightarrow
\overline{R\ot S}=
(\bar{R}\ot \bar{S})\oplus 
(\bar{R}\ot k1_S)\oplus (k1_R\ot \bar{S})
\, $ to middle tensor components
and that $\pr_2=\pr^{}_{\BB_R}\ot \pr^{}_{\BB_S}$,
see \cref{ReducedTwistedTensorProduct} and \cref{DifferentialReduced}.
By definition of the action,
$\pr_1$ and $\inc_1$ are $\ttpp$-bimodule homomorphisms
with $\pr_1\, \inc_1=1$, the identity map.
Also note that $\pr_2$ is an $\ttpp$-bimodule homomorphism
since $\tau$ is the swap (interchange) map on $R\ot 1_S$ and $1_R\ot 
S$, but $\inc_2$ may not be: the function
$\inc_2=\inc^{}_{\BB_R}\ot \inc^{}_{\BB_S}$ is merely a $k$-vector space map
with $\pr_2\,  \inc_2=1$.

\subsection*{Defining $\AWt$ and $\EZt$
  in terms of inclusion and projection}
 We define 
  $$\ \AWt:\BB_{R \ott S} \longrightarrow  \BB_R\ott \, \BB_S
\qquad\text{ by }   \AWt=\AWtBB= \pr_2\ \phi\ \varrho\ \inc_1\, ,
$$
  for twisted perfect unshuffle map $\varrho$ 
of~(\ref{twistedperfectunshufflemap})
and front/back face map  $\phi$ of \cref{frontbackmap},
and
$$ \ \ \ \EZt: \BB_R\ott \, \BB_S \longrightarrow\BB_{\ttp}
\quad\  \  \text{ by } \EZt=\EZtBB=\pr_1\ \varrho^{-1}\  \theta\, 
\ \inc_2\, ,
$$
for shuffle map $\theta$ of \cref{shufflingsummap}, 
so that $\AW_{\BB}^{\tau} = \pr_2 \AW_{\B}^{\tau}\inc_1$
and $\EZ^{\tau}_{\BB}=\pr_1\EZ^{\tau}_{\B}\inc_2$:
\begin{small}
  $$
\begin{tikzcd}[row sep=large, column sep=10ex, text height=1.5ex, text depth=0.25ex]
  \BB_{\ttps}
 \arrow[hook]{r}{\inc_1}
 \arrow[bend left=15, shift left=1.5ex]{rrrr}{\AWt}
  &    \B_{\ttps}
  \arrow[shift left=0ex]{r}{\  \varrho}
   \arrow[two heads, shift left=1.5ex]{l}{\pr_1}
   &[1em]
   Y
  \arrow[shift left=0ex]{r}{\  \phi\ }
   \arrow[shift left=1.5ex]{l}{\,  \varrho^{-1} \ }    
 & \B_R\ott \, \B_S
 \arrow[two heads]{r}{ \pr_2 }
 \arrow[shift left=1.5ex]{l}{\theta}
 & \BB_R\ott \, \BB_S
 \arrow[hook, shift left=1.5ex]{l}{\ \inc_2}
 \arrow[bend left=15, shift left=3ex]{llll}{{\EZt}}
 \ . 
 \end{tikzcd}
 $$
\end{small}%
The map $\AWt$ is an $(\ttp)$-bimodule homomorphism
since $\pr_2$, $\phi$, $\varrho$, and $\inc_1$
are all $(\ttp)$-bimodule maps.
We will verify that $\EZt$ is an $(\ttp)$-bimodule homomorphism.

\subsection*{Complements to the reduced versions of resolutions}
We first argue that $\pr_1\, \EZtB$ is zero on
$\ker\pr_2\subset \B_R\ott \B_S$,
a vector space complement to $\BB_R\ott \BB_S$.
For fixed degree $n$, consider the map $\pr_1\, \varrho^{-1}\, \theta$
on $(\B_R\ott \B_S)_n$.
The shuffle map $\theta$ (see \cref{shufflingsummap})
is a direct sum over $0\leq \ell\leq n$
of maps
taking an  element  
\begin{equation}\label{eta}
  z=1_R\ot r_1\ot\cdots\ot r_{\ell}\ot 1_R \ot 
  1_S\ot s_{\ell+1}\ot\cdots \ot s_{n}\ot 1_S 
  \quad\text{
    for $r_i\in R$ and $s_i\in S$}
  \end{equation}
  to a signed sum over shuffles
  $\sigma^{-1}\in \mathfrak{S}_{\ell, n-\ell}\subset \mathfrak{S}_n$
of terms of the form
  \begin{equation}
    \label{oneshuffle}
  \begin{aligned}
  \theta_{\sigma^{-1}}(z) &=1_R\ot \, \sigma^{-1}(r_1\ot\cdots\ot r_n) \ot 1_R 
  \ot  1_S\ot \sigma^{-1}(s_1\ot\cdots\ot s_n)\ot 1_S
  \\
  &= 1_R\ot r_{\sigma(1)}\ot\cdots\ot r_{\sigma(n)}\ot 1_R \ot 
  1_S\ot s_{\sigma(1)} \ot\cdots\ot s_{\sigma(n)}\ot 1_S
  \, ,
\end{aligned}
\end{equation}
where we set $r_{\ell+1}=\cdots = r_n=1_R$ and 
$s_1=\cdots = s_{\ell}=1_S$.
Notice that for each $i$ and $\sigma$, we have 
$r_{\sigma(i)}=1_R$ or $s_{\sigma(i)}=1_S$.
As $\tau$ is just the swap (interchange) map on $R\ot 1_S$ and $1_R\ot S$,
the image of such an element under the twisted unshuffling map
$\varrho^{-1}$
lies in the $k$-span of elements of the form
\begin{equation}  \label{InTheSpan}
  1_{\ttp}\ot (r_1'\ot s_1') \ot \cdots\ot (r_n' \ot s_n')\ot  1_{\ttp}
\qquad\text{ in } \B_{\ttp}
\end{equation}
 with $(r_i'\ot s_i')$ in $R\ott S$
and
 $r_i'=1_R$ or $s_i' = 1_S$ for each $i$.
 
If our original element $z$ of \cref{eta}
lies in $\ker\pr_2\subset \B_R\ott \B_S$, then
$r_m$ lies in $k1_R$ for some $m\in \{1,\ldots,\ell\}$
or
$s_m$ lies in $k1_S$ for some $m\in \{\ell+1,\ldots,n\}$.
Hence for any $\sigma^{-1}$ in $\mathfrak{S}_{\ell, n-\ell}$,
there is some $m\in \{1,\ldots,n\}$ with both 
$r_{\sigma(m)}$ and $s_{\sigma(m)}$ scalars,
  and hence $\varrho^{-1}\theta(z)$
  lies in the $k$-span of elements of the form
  (\ref{InTheSpan}) with $r_m'\ot s_m'=1_R\ot 1_S$.
Thus $\pr_2\, \varrho^{-1}\, \theta(z)=0$.

    As $\pr_1$, $\varrho^{-1}$, and $\theta$
  are $\ttpp$-bimodule maps,
and pure tensors of the form $z$ in~Eq.~(\ref{eta}) lying in $\Ker \pr_2$
generate $\ker\pr_2$ as an $\ttpp$-bimodule,
\begin{equation}\label{EZOnComplement}
 \pr_1\,\EZtB= \pr_1\,\varrho^{-1}\,\theta\equiv 0
\qquad\text{ on $\ker \pr_2=\ima (1-\inc_2\pr_2) \subset \B_R\ott \B_S$.}
\end{equation}

We likewise argue that
\begin{equation}\label{AWOnComplement}
  \pr_2\, \AWtB=
  \pr_2\, \phi\, \varrho\equiv 0 
  \qquad\quad\text{  on $\ker \pr_1=\ima (1-\inc_1\pr_1)\subset \B_{\ttp}$,}
  \end{equation}
i.e., $\pr_2\, \AWtB$
vanishes on
a vector space complement to $ \BB_{\ttp}$.
Indeed, say  
$$
z=(1_{R}\ot 1_S)\ot (r_1\ot s_1)\ot\cdots\ot
(r_n\ot s_n)\ot (1_R\ot 1_S)
\quad\text{
    for $r_i\in R$ and $s_i\in S$}
$$
lies in $\ker\pr_1$ with $(r_i\ot s_i)$ in $\ttp$ for all $i$,
say
$r_m\ot s_m=1_R\ot 1_S$
for some fixed $m$.
Then $\varrho$ takes $z$ to a sum of elements of the form
\begin{equation}
  z'=1_R\ot r_1'\ot\cdots\ot r_{n}'\ot 1_R \ot 
  1_S\ot s_1'\ot\cdots \ot s_{n}'\ot 1_S  
  \end{equation}
  with $r_m'=r_m=1_R$ and $s_m'=s_m=1_S$.
  The map $\phi$ in turn takes $z'$ to a sum over $1\leq\ell\leq n$
of  elements of the form
  $$
(r_1'\cdots r_{\ell}' ) \ot r_{\ell+1}' \ot\cdots\ot r_n'
  \ot 1_R\ot 1_S\ot
  s_1'\ot \cdots\ot s_{\ell}'\ot (s_{\ell+1}'\cdots s_n')
 $$
which projects to zero under
$\pr_2$
since
among the inner tensor components
(ignoring $1_R\ot 1_S$ in the 
middle) is found either
$s_m'=1_S$  (when
$m\leq \ell$) or $r_m'=1_R$
(when $m>\ell$).
The claim \cref{AWOnComplement} then follows
as pure tensors of the form $z$ generate $\ker\pr_1$ as an $\ttpp$-bimodule 
and $\pr_2$, $\phi$, and $\varrho$ are $\ttpp$-bimodule maps.

\subsection*{$\EZt$ is a bimodule map}
To see that
$\EZtBB$ is an $(\ttp)$-bimodule homomorphism
even though $\inc_2$
may not be,
let $\rho_{_\BB}$, $\rho_{_\B}$, and $\rho'$ denote
the maps recording the bimodule action of $\ttp$
on $\BB_R\ott\BB_S$, $\B_R\ott \B_S$, and $\BB_{\ttp}$,
respectively, noting that
$\rho_{_{\BB}}=\pr_2\,\rho_{_{\B}}\, (1\ot \inc_2\ot 1)$
on $\ttpp\ot (\BB_R\ott\BB_S)\ot \ttpp$
by Diagram~(\ref{bimodstructure}) and
\cref{iterativetwistingreduced}.
Then by
\cref{EZOnComplement},
$$\begin{aligned}
  \EZtBB \, \rho_{_\BB}
  &= (\pr_1 \, \EZtB\, \inc_2)\, (\pr_2\,\rho_{_{\B}})\, (1\ot \inc_2\ot 1)   
  \\   &=
  \pr_1\, \EZtB\,\rho_{_{\B}}\, (1\ot \inc_2 \ot 1)
  +
  \pr_1\, \EZtB\, (\inc_2\, \pr_2-1)\,\rho_{_\B}\, (1\ot \inc_2\ot 1)
    \\   &=
    \rho'\ (1\ot \pr_1\, \EZtB \ot 1) \, (1\ot \inc_2\ot 1)
 \\   &=
\rho'\, (1\ot \EZtBB\ot\, 1)
\end{aligned}
$$
as $\pr_1$ and $\EZtB$ are $\ttpp$-bimodule  
maps.
Thus $\EZtBB$ is an $\ttpp$-bimodule map.

\subsection*{Chain maps}
We argue that $\AWtBB$ and $\EZtBB$ 
are chain maps 
using the fact that (the unreduced versions) $\AWtB$ and $\EZtB$ are 
chain maps by \cref{botharechainmaps}.
Consider the differentials
$d^{}_{\BB}$, $d^{}_{\B}$, $\del^{}_{\BB}$, and $\del^{}_{\B}$ on
$\BB_{\ttp}$, $\B_{\ttp}$, $\BB_R\ott\BB_S$,
and $\B_R\ott\B_S$, respectively.
We again write $1$ for the identity map with domain given by context:
$$
\begin{aligned}
\EZtBB\, \del^{}_{\BB}
=
(\pr_1\, \EZtB\, \inc_2)\, (\pr_2\, \del^{}_{\B}\, \inc_2)
=
\pr_1\, \EZtB\, \del^{}_{\B}\, \inc_2 
+
\pr_1\, \EZtB\, (\inc_2\, \pr_2-1)\, \del^{}_{\B}\, \inc_2
\, .
\end{aligned}
$$
The second summand is 
zero by \cref{EZOnComplement},
and  \cref{botharechainmaps} and \cref{ProjectD} then imply that 
$$
\begin{aligned}
 \EZtBB\, \del^{}_{\BB}
 &=
\pr_1\, d^{}_{\B}\, \EZtB\, \inc_2 
\\ &=
(\pr_1\, d^{}_{\B}\, \inc_1)\, (\pr_1\, \EZtB\, 
\inc_2) 
+
\pr_1\, d^{}_{\B}\, (1 - \inc_1\, \pr_1)\, \EZtB\, \inc_2
 &=\ 
d^{}_{\BB}\,  \EZtBB
\, .
\end{aligned}
$$
Thus $\EZt$ is a chain map.
Likewise,
one may use \cref{botharechainmaps} and \cref{ProjectD}
to show that $\AWt$ is a chain map,
as $\del^{}_{\BB}=\pr_2 \,\del^{}_{\B}\,\inc_2
\equiv 0$ on 
$\ker \pr_2=\Ima (\inc_2\,\pr_2 - 1) $ by 
\cref{AWOnComplement}. 
 
\subsection*{The composition is the identity}
Lastly, we argue that $\AWtBB\, \EZtBB= 1$.
By \cref{AWOnComplement},
$$
\begin{aligned}
  \AWtBB\ \EZtBB
  &=   
  (\pr_2\, \AWtB\, \inc_1)\ (\pr_1\, \EZtB\, \inc_2)
\\  &=
\pr_2\, \AWtB\, \EZtB\, \inc_2 
+
\pr_2\, \AWtB\, (\inc_1\, \pr_1-1)\, \EZtB\, \inc_2
\\  &=
\pr_2\, (\phi\, \varrho)\, (\varrho^{-1}\, \theta)\, \inc_2
\\ &=
\pr_2\, \phi\, \theta\, \inc_2
\, .
\end{aligned}
$$
We show directly that $\pr_2\, \phi\, \theta\, \inc_2$ is
the identity map on $(\BB_R\ott\BB_S)_n$ for each $n$.
This map is a signed sum over
$1\leq \ell', \ell \leq n$
and over
$(\ell, n-\ell)$-shuffles
 $\sigma^{-1}$ 
 of maps
 taking
  $$
  z=1_R\ot r_1\ot\cdots\ot r_{\ell}\ot 1_R \ot
  1_S\ot s_{\ell+1}\ot\cdots \ot s_{n}\ot 1_S
  \quad\text{
    for $r_i\in \bar{R}$ and $s_i\in \bar{S}$}
  $$
in $(\BB_{R} \ott \BB_S )_n$
 to
  $$(r_{\sigma(1)}\cdots r_{\sigma(\ell')})\ot r_{\sigma(\ell'+1)}\ot \cdots \ot r_{\sigma(n)}\ot 1_R
  \ot 1_S\ot s_{\sigma(1)}\ot \cdots\ot s_{\sigma(\ell')} \ot (s_{\sigma(\ell'+1)}\cdots
    s_{\sigma(n)})\, ,
    $$
    where again we set $r_{\ell+1}=\cdots = r_n=1_R$ and 
    $s_1=\cdots = s_{\ell}=1_S$ (see \cref{oneshuffle} and 
\cref{frontbackmap}).
This expression projects to zero under $\pr_2$ 
unless $\ell'=\ell$ and 
  $\theta_{\sigma^{-1}}$
  takes $z$
    to
  $$1_R\ot\cdots\ot 1_R\ot r_1\ot\cdots\ot r_{\ell}\ot 1_R \ot
  1_S\ot s_{\ell+1}\ot\cdots \ot s_{n}\ot 1_S\ot \cdots\ot 1_S
  $$
 in $R^{\ot n}\ot S^{\ot n}$, in which case the expression is just $z$
 again.
  There is only one such shuffle $\sigma$
  for fixed $\ell$ with this property, and its 
  sign is $(-1)^{\ell(n-\ell)}$ as required to imply that
  $\pr_2 \, \phi \, \theta\, \inc_2$ fixes $z$.
  Then as pure tensors of the form $z$ generate 
  $(\BB_R\ott \BB_S)_n$ as an $\ttpp$-bimodule,
  the composition $\AWtBB\, \EZtBB$ is the
 identity map on $ (\BB_R\ott \BB_S)_n$
 for each $n$.
   \end{proof}

   \vspace{2ex}

  \begin{example}\label{LieAlgebraExample}{\em
      Any nonabelian $2$-dimensional Lie algebra
      $\mathfrak{g}$ has universal enveloping algebra
     $\mathcal{U}(\mathfrak{g})$ isomorphic to 
      $$A=k\langle x, y: yx-xy-x \rangle
      \cong R\ott S
      $$
      for $R=k[x]$, $S=k[y]$, and twisting map
      $\tau: S\ot R \rightarrow R\ot S $
      defined by $y\ot x \mapsto x\ot y + x\ot 1$.
      We identify $\bar R$, $\bar S$ with the $k$-span
      of the elements of positive degree in $k[x]$, $k[y]$,
      respectively,
      for $x$ and $y$ in degree $1$.
      We likewise identify $\bar A$ with the $k\text{-span}$
      of elements of positive degree under the vector space
      isomorphism $A\cong k[x,y]$ given by the
      Poincar\'e-Birkhoff-Witt Theorem.
                     Set
            $$
            \begin{aligned}
      a&=\ \ \ 1_A\ot y^2\ot x\ot x\ot 1_A , 
      \\
      b&=\ \ \  1_R\ot x\ot x \ot 1_R\ot 1_S \ot y^2\ot 1_S
      +4\cdot 1_R\ot x\ot x \ot 1_R\ot 1_S \ot y\ot 1_S,
         \\
         c&=\ \ \ 
         1_A\ot x\ot x\ot y^2\ot 1_A + 4\cdot 1_A\ot x\ot x\ot y\ot 1_A
         + 1_A\ot y^2\ot x\ot x\ot 1_A\\
        & \quad \quad  -1_A\ot x\ot y^2\ot x\ot 1_A
         -2\cdot 1_A\ot x\ot y\ot x\ot 1_A
 \, ,
\end{aligned}
$$
for $a$, $c$ in  $(\BB_A)_3=A\ot \bar A\ot \bar
A\ot \bar A \ot A$
and $b$ in $(\BB_{R})_2\ot (\BB_{S})_1=R\ot \bar R\ot \bar R\ot 
R\ot S\ot \bar S\ot S$. Then
$$\AWt(a)=b,\quad \EZt(b)=c,
\quad\text{ and }\AWt(c)=b
$$
so
$\AWt \, \EZt(b)=b$
but
$
\EZt \, \AWt(a)\neq a$.
}
 \end{example}

 \vspace{2ex}
 
\section{Compatible chain maps}\label{sec:compatiblechainmaps}

Before using the twisted Alexander-Whitney and
Eilenberg-Zilber maps 
to convert between resolutions,
we make precise here the idea of a chain map
preserving a twisting map
$\tau: S\ot R\rightarrow R\ot S$ for $k$-algebras $R$ and $S$.
We show that this compatibility with twisting allows us to extend
chain maps to twisted tensor product resolutions. 

\vspace{2ex}

\begin{definition}{\em
    We say a chain map $\psi:C\sd\rightarrow C'\sd$
    of bimodule resolutions $C\sd$, $C'\sd$ of $R$,
    $$ 
\hspace{5ex}
\mbox{
\entrymodifiers={+!!<0pt,\fontdimen22\textfont2>}
\xymatrixcolsep{6ex}
\xymatrixrowsep{6ex}
\xymatrix{
  \cdots \ar[r]
  & 
  C_{1}
\ar[r]
\ar[d]_{\psi_1}
 &  
 C_{0}  \ar[d]_{\psi_0}
 \ar[r]
&  
R\ar[r]
\ar@{=}[d]
& 
0
\\
\cdots \ar[r]
& 
C'_{1}
\ar[r]
 & 
 C'_{0}
 \ar[r]
& 
R  \ar[r]
& 
0\ ,
}
}
$$
    {\em is compatible with} the twisting map $\tau$
    if $C\sd$ and $C'\sd$
    are
    both compatible
    with $\tau$
    via chain maps
    $\tau_{_C}: S\ot C\sd\rightarrow C\sd\ot S$
    and     $\tau_{_{C'}}: S\ot C'\sd\rightarrow C'\sd\ot S$
    (see \cref{CompatibleResolutionsDef})
    satisfying
  $$
  \begin{aligned}
    (\psi\otimes 1_S)\, \tau^{}_{_C}
    &=\tau^{}_{_{C'}}\, (1_S\ot \psi)
  &&\quad\text{ as maps }
    S\ot C\sd\longrightarrow C'\sd
    \ot S
   \, .
  \end{aligned}
  $$
 We likewise say a 
 chain map $\psi:D\sd\rightarrow D'\sd$ between bimodule resolutions of $S$
 is compatible with $\tau$
 when it is compatible with $\tau^{-1}$. 
    }
  \end{definition}

  \vspace{2ex}

\begin{remark}
  \label{compatiblechainmaps}
  {\em
    Note that the right-left distinction here again is cosmetic as $\tau$
is invertible: 
A chain map $\psi:D\sd\rightarrow D'\sd$ between bimodule resolutions of $S$
 is compatible with $\tau$
 precisely when
 $D\sd$ and $D'\sd$ are compatible with $\tau$ via chain maps 
 $\tau_{_D}: D\sd\ot R\rightarrow R\ot D\sd$ and 
 $\tau_{_{D'}}: D'\sd\ot R\rightarrow R\ot D'\sd$ with 
$$
  \begin{aligned}
   (1_R\ot \psi)\, \tau^{}_{_D}
  &=\tau^{}_{_{D'}}\, (\psi\ot 1_R)  
  &&\quad\text{ as maps }
  D\sd\ot R\longrightarrow R\ot D'\sd
  \, . 
  \end{aligned}
  $$
   Thus chain maps $\psi_{_R}:C\sd\rightarrow C'\sd$ and $\psi_{_S}: D\sd\rightarrow D'\sd$
    between compatible resolutions 
    are compatible with the twisting map $\tau$
    when the following diagrams are commutative:
\begin{equation*}\label[diagram]{ChainMapCompatibleDiagram}
\hspace{5ex}
\mbox{
\entrymodifiers={+!!<0pt,\fontdimen22\textfont2>}
\xymatrixcolsep{6ex}
\xymatrixrowsep{6ex}
\xymatrix{
S\ot C\sd
\ar[rr]^{\tau^{}_{_C}}   \ar[d]_{1\ot\,\psi_{_R}}
 &  & 
C\sd\ot S  \ar[d]^{\psi_{_R}\,\ot 1}
&  & 
 \\
S\ot C'\sd
\ar[rr]_{\tau^{}_{C'}}
 & & 
C'\sd\ot S
 & &
}
}
\hspace{-7ex}
\mbox{
\entrymodifiers={+!!<0pt,\fontdimen22\textfont2>}
\xymatrixcolsep{6ex}
\xymatrixrowsep{6ex}
\xymatrix{
D\sd\ot R
\ar[rr]^{\tau^{}_{_D}}  \ar[d]_{\psi_{_S}\,\ot 1}
 &  & 
R\ot D\sd
\ar[d]^{1\ot\,\psi_{_S}}
&  & 
 \\
D'\sd\ot R
\ar[rr]_{\tau^{}_{_{D'}}}
 & & 
R\ot D'\sd
 \ . &  &
}
}
\end{equation*}
}
\end{remark}

\vspace{2ex}

Compatible chain maps extend to give chain maps on twisted tensor
product resolutions.
\begin{lemma}\label{lem:alphaRalphaS}
If
$\psi_{_R}: C\sd\rightarrow C'\sd$ and $\psi_{_S}: D\sd \rightarrow D'\sd$
are chain maps of $R$-bimodule resolutions
of $R$ and $S$, respectively,
 compatible
 with the twisting map $\tau$,
 then
 $$\psi_{_R}\ot \psi_{_S}: C\sd\ott D\sd\longrightarrow C'\sd \ott D'\sd
  $$
is
 a chain map of bimodule resolutions
  of $ R\ott S$:
  $$ 
\hspace{5ex}
\mbox{
\entrymodifiers={+!!<0pt,\fontdimen22\textfont2>}
\xymatrixcolsep{6ex}
\xymatrixrowsep{6ex}
\xymatrix{
  \cdots \ar[r]
  & 
  (C\sd\ott D\sd)_{1}
\ar[r]
\ar[d]_{\psi_{_R}\ot\, \psi_{_S}}
 &  
 (C\sd\ott D\sd)_{0}  \ar[d]_{\psi_{_R}\ot\, \psi_{_S}}
 \ar[r]
&  
R\ott S \ar[r]
\ar@{=}[d]
& 
0
\\
\cdots \ar[r]
& 
(C'\sd\ott D'\sd)_{1}
\ar[r]
 & 
 (C'\sd\ott D'\sd)_{0}
 \ar[r]
& 
R\ott S  \ar[r]
& 
0
\, .
}
}
$$
  \end{lemma}
\begin{proof}
Say $C\sd$, $C'\sd$, $D\sd$, and $D'\sd$ 
are compatible with $\tau$ via chain  
maps $\tau_{_C}$, $\tau_{_{C'}}$, $\tau_{_D}$, and $\tau_{_{D'}}$, respectively,
with the diagrams of \cref{compatiblechainmaps} commutative.
 The following diagram gives the bimodule action of
$\ttp$ before and after application
of the map $\psi_{_R}\ot \psi_{_S}$,
where again we use maps $\rho_{_C}$, $\rho_{_{C'}}$,
$\rho_{_D}$, and $ \rho_{_{D'}}$
to record the bimodule action (of $R$ or $S$) on $C\sd$, $C'\sd$, $D\sd$, and
$D'\sd$, respectively:
\begin{small}
$$
\entrymodifiers={+!!<0pt,\fontdimen22\textfont2>}
\xymatrixcolsep{20ex}
\xymatrixrowsep{8ex}
\xymatrix{
     R\! \ot \! S \! \ot\!  C\sd \ot \! D\sd\ot \! R\! \ot \! S  
    \ar[d]_{\hspace{2ex}1\ot\, \tau^{}_{_C}\ot\, \tau^{}_{_D}\, \ot 1 \ \ }
    \ar[r]_{\hspace{2ex} 1\ot 1 \ot \psi_{_R}\ot \psi_{_S} \ot 1 \ot 1\ \ }
&
    R\! \ot \! S \! \ot\!  C'\sd \ot \! D'\sd\ot \! R\! \ot \! S  
    \ar[d]^{\hspace{2ex} 1\ot\, \tau^{}_{_{C'}}\ot\, \tau^{}_{_{D'}}\, \ot 1\ \ }
    \\
  R\! \ot \! C\sd\ot \! S\! \ot \! R\! \ot \! D\sd \ot\! S 
  \ar[d]_{\hspace{2ex}1\ot 1\ot \tau \ot 1\ot 1\ \  }
  \ar[r]_{\hspace{2ex} 1\ot \psi_{_R}\ot 1 \ot 1 \ot  \psi_{_S}  \ot 1\ \ }
  &
  R\! \ot \! C'\sd\! \ot \! S\! \ot \! R\! \ot \! D'\sd\! \ot\! S 
  \ar[d]^{\hspace{2ex}1\ot 1\ot \tau \ot 1\ot 1\ \ }
   \\
            R \! \ot \! C\sd\ot\! R\! \ot \! S\! \ot\!  D\sd\ot\! S 
            \ar[d]_{\hspace{2ex}\rho^{}_{_C}\ot\, \rho^{}_{_D}\ \ }
           \ar[r]_{\hspace{2ex} 1\ot \psi_{_R}\ot 1 \ot 1 \ot  \psi_{_S}  \ot 1\ \ }
           &
            R \! \ot \! C'\sd \ot\! R\! \ot \! S\! \ot\!  D'\sd \ot\! S 
            \ar[d]^{\hspace{2ex}\rho^{}_{_{C'}}\ot\, \rho^{}_{_{D'}}\ \ }
              \\
           C\sd\ot D\sd
              \ar[r]_{\hspace{2ex} \psi_{_R}\ot\,  \psi_{_S} \ \ }
            &
            C'\sd\ot D'\sd
             \, .
} 
$$
\end{small}\noindent
The top square is commutative since $\psi_{_R}$ and $\psi_{_S}$
are compatible with $\tau$,
the middle square is commutative because the given maps apply to different
tensor components, and
the bottom square is commutative because 
$\psi_{_R}$ and $\psi_{_S}$
are bimodule homomorphisms.
Hence $\psi_{_R}\ot \psi_{_S}$ is an $\ttpp$-bimodule map
and the result follows
as the differentials on $C\sd\ott D\sd$ and $C'\sd\ott D'\sd$
are just the usual total differentials.
\end{proof}

\vspace{2ex}

\begin{remark}\label{rk:HKH}
{\em
  For construction of Hopf-Koszul Hecke 
  algebras in \cite{SW-DeformationTheoryHopfActions},
    we take the chain map $\psi_{_S}$ to be inclusion
    of the Koszul complex into the reduced bar complex
    and $\psi_{_R}$ to be the identity map on the reduced bar complex
    of a Hopf algebra.  See \cref{sec:Hopf-Koszul}.
        }
\end{remark}

\section{Conversion between Resolutions
using compatible surjective and injective maps}
\label{sec:conversion}
In this section, we combine the twisted Alexander-Whitney and
Eilenberg-Zilber maps for reduced bar resolutions
with compatible chain maps to and from other resolutions in order 
to transfer explicit homological information.
In the next section, we consider the case when
compatible chain maps in only one direction are given.

\subsection*{Conversion theorem}
Again we fix $k$-algebras $R$ and $S$ and a twisting map 
$\tau: S\ot R\rightarrow R\ot S$ recording an 
algebra structure on $R\ot S$. 
We take bimodule resolutions
\begin{equation*}
\begin{aligned}
&C\sd : \  \cdots \, \longrightarrow C_1
\longrightarrow C_0 \longrightarrow 0
\ \ \text{ and }
&&C'\sd: \  \cdots \, \longrightarrow C'_1 \longrightarrow C'_0 \longrightarrow 0 
  \ \ \text{ of }\ \ R\, ,  \text{ and } 
\\
&D\sd :\  \cdots \longrightarrow D_1 \longrightarrow D_0
\longrightarrow 0\ \,
\text{ and }
&&D'\sd:\  \cdots \longrightarrow D'_1 \longrightarrow D'_0 \longrightarrow 0
\ \text{ of }\ \  S \, , 
\end{aligned}
\end{equation*}
compatible with $\tau$
(see \cref{CompatibleResolutionsDef})
and consider the
twisted product resolutions $C\sd\ott D\sd$
and $C'\sd\ott D' \sd$
(see \cref{twistedproductresolution}).
Often one may embed and project resolutions, so we begin
with the case when chain maps exist that
behave as projection and inclusion:
 $$
\hspace{5ex}
\mbox{
\entrymodifiers={+!!<0pt,\fontdimen22\textfont2>}
\xymatrixcolsep{6ex}
\xymatrixrowsep{6ex}
\xymatrix{
C'_{n+1}
\ar[rr]^{\partial_{n+1}'}   \ar[d]_{\pi_{_R}}
 &  & 
C'_{n}   \ar[d]_{\pi_{_R}}
&  & 
 \\
C_{n+1} 
 \ar[rr]^{\partial_{n+1}}  \ar@<-1ex>[u]_{\iota_{_R}}  
 & & 
\ C_{n} 
\ar@<-1ex>[u]_{\iota_{_R}} 
 & &
}
}
\hspace{-8ex}
\raisebox{-.75cm}[\height][\depth]{and}
\hspace{2ex}
\mbox{
\entrymodifiers={+!!<0pt,\fontdimen22\textfont2>}
\xymatrixcolsep{6ex}
\xymatrixrowsep{6ex}
\xymatrix{
D'_{n+1}
\ar[rr]^{\partial_{n+1}'}  \ar[d]_{\pi_{_S}}
 &  & 
D'_{n}
\ar[d]_{\pi_{_S}}
&  & 
 \\
D_{n+1} 
 \ar[rr]^{\partial_{n+1}}  \ar@<-1ex>[u]_{\iota_{_S}}  
 & & 
\ D_{n}
\ar@<-1ex>[u]_{\iota_{_S}} 
 \ . &  &
}
}
$$
\vspace{1ex}

\cref{lem:alphaRalphaS} implies the following.

\begin{thm}\label{ConversionEverythingCompatible}
Suppose $R$-bimodule resolutions $C\sd,C'\sd$ of $R$ and 
 $S$-bimodule resolutions $D\sd,D'\sd$ of $S$
  are compatible with the twisting map $\tau$ and
     $$\pi_{_R}: C'\sd\rightarrow C\sd, \quad
  \iota_{_R}: C\sd\rightarrow C'\sd, \quad
  \pi_{_S}: D'\sd\rightarrow D\sd, \quad
  \iota_{_S}: D\sd \rightarrow D'\sd
  $$
  are compatible chain maps satisfying $\pi_{_R}\, \iota_{_R}=1_C$ and
  $\pi_{_S}\, \iota_{_S}=1_D$.
 Then 
$$
\pi_{_R}\ot\pi_{_S}:C'\sd\ott D'\sd 
\longrightarrow C\sd \ot _\tau D\sd 
\quad\text{ and }\quad
\iota_{_R}\ot \iota_{_S}: C\sd \ot _\tau D\sd \longrightarrow 
C'\sd\ott D'\sd\, 
\, 
$$
are chain maps of bimodule resolutions 
 of $\ttp$
with composition giving the identity:
$$(\pi_{_R}\ot \pi_{_S})\  (\iota_{_R}\ot \iota_{_S})  
= 1_{C\ott D}\, . 
$$
If  in addition $R$ and $S$ are graded algebras and
$\pi_{_R}$, $\pi_{_S}$, $\iota_{_R}$, $\iota_{_S}$ are graded chain maps
of graded resolutions,
then $\pi_{_R}\ot\pi_{_S}$ and $\iota_{_R}\ot \iota_{_S}$ are graded chain maps as well.
\end{thm}

\subsection*{Conversion theorem for bar resolutions}

We now use the twisted
Alexander-Whitney and Eilenberg-Zilber maps constructed
in \cref{AW-EZforreducedbar}
to connect
the default reduced bar resolution for a twisted tensor
product (used for defining algebraic structures)
to another choice of twisted product resolution for that algebra
(which may be tractable for computations).
We obtain a 
splitting under some conditions,
\vspace{3ex} 
\begin{equation*}
\xymatrix{
  &  C\sd\ot_\tau D\sd
  \ar[rr] &
  & \BB_{R\ot_{\tau} S}\ .
  \ar@(ul,ur)@{-->}[ll] \, 
}
\end{equation*}

\vspace{1ex}

\begin{thm}\label{barconversion}
  Suppose bimodule resolutions $C\sd$ and $D\sd$ of $R$ and $S$,
  respectively, are compatible with the twisting map $\tau$
 and there exist  compatible
  chain maps
  $$\pi_{_R}: \BB_R\rightarrow C\sd,\quad
  \iota_{_R}: C\sd\rightarrow \BB_R,\quad
  \pi_{_S}: \BB_S\rightarrow D\sd,\quad
  \iota_{_S}: D\sd \rightarrow \BB_S$$
  with $\pi_{_R}\, \iota_{_R}=1_C$ and $\pi_{_S}\, \iota_{_S}=1_D$.
Then there exist chain maps
$$
\pi:\BB_{R\ot_{\tau} S}   
\longrightarrow C\sd \ot _\tau D\sd 
\quad\text{ and }\quad
\iota: C\sd \ot _\tau D\sd \longrightarrow 
\BB_{R\ot_{\tau} S}
\, 
$$
of  bimodule resolutions of $\ttp$
with composition giving the identity:
$$
\pi \, \iota 
= 1_{C \ot_{\tau} D}\, .
$$
If  in addition $S$ and $R$ are graded algebras and
$\pi_{_R}$, $\pi_{_S}$, $\iota_{_R}$, $\iota_{_S}$ are graded chain maps
of graded resolutions,
then $\iota$ and $\pi$ are graded chain maps as well.
\end{thm}
$$
\entrymodifiers={+!!<0pt,\fontdimen22\textfont2>}
\xymatrixcolsep{2ex}
\xymatrixrowsep{8ex}
\xymatrix{
\ar@(dl, ul)[dd]_{\pi} & &
(\BB_{R\otimes_{\tau} S})_{n+1}   
\ar[rr]  \ar@<-1ex>[d]_{\rm{AW}^\tau}
&  & 
 (\BB_{R\ot_{\tau} S})_{n}
\ar@<-1ex>[d]_{\rm{AW}^\tau}
&  
 \\
& & 
(\BB_{R} \otimes_{\tau} \BB_S )_{n+1} 
 \ar[rr]  \ar[u]<-1ex>_{\rm{EZ}^\tau}
 \ar@<-1ex>[d]_{\pi_{_R}\ot\, \pi_{_S}}  
 & & 
(\BB_{R} \otimes_{\tau} \BB_S )_{n} 
\ar[u]<-1ex>_{\rm{EZ}^\tau}  
\ar@<-1ex>[d]_{\pi_{_R}\ot\, \pi_{_S}}  
&  & 
  \\ 
  & 
&
(C\sd \otimes_{\tau} D\sd )_{n+1} 
 \ar[rr]  \ar[u]<-1ex>_{\iota_{_R} \ot\,  \iota_{_S}}
 & & 
(C\sd \otimes_{\tau} D\sd )_{n} 
\ar[u]<-1ex>_{\iota_{_R} \ot\,  \iota_{_S}}
& 
\ar@(ur,dr)[uu]_{\iota}
}
$$
\begin{proof}
  \cref{ConversionEverythingCompatible}
  gives chain maps 
  $$ \pi_{_R}\ot \pi_{_S}: \BB_R \ott \BB_S 
\longrightarrow 
C\sd\ott D\sd
  \qquad\text{ and }\qquad
\iota_{_R}\ot \iota_{_S}: 
  C\sd\ott D\sd \longrightarrow \BB_R\ott\BB_S   
$$
of bimodule resolutions of $\ttp$.
We define the compositions
\begin{equation}\label{eqn:iota-pi}
  \pi= (\pi_{_R}\ot\pi_{_S})\,  \AW^\tau 
       \ \ \ \mbox{ and } \ \ \  
  \iota = \EZ^\tau \,  (\iota_{_R}\ot \iota_{_S})
 \end{equation}
 using the twisted Alexander-Whitney
 and Eilenberg-Zilber maps of \cref{AW-EZforreducedbar}:
 \begin{equation*}
  \begin{aligned}
{\rm{AW}}^{\tau}:\ \ \BB_{\ttps}
\longrightarrow \BB_R\ot_{\tau} \, \BB_S \, 
\quad\text{ and }\quad 
  \BEZ^{\tau}:\ \ 
\BB_R\ot_{\tau} \, \BB_S \longrightarrow\BB_{\ttps}
\, .
\end{aligned}
\end{equation*}
As the maps
${\BAW}^{\tau}$, $\BEZ^{\tau}$, $\pi_{_R}\ot \pi_{_S}$, and $\iota_{_R}\ot\iota_{_S}$
are chain maps of bimodule resolutions of $R\ott S$, so are
$\pi$ and $\iota$,
and 
the composition is the identity in each degree
by \cref{AW-EZforreducedbar}:
$$\pi\, \iota
=(\pi_{_R}\ot \pi_{_S})\, \AW^\tau 
\, \EZ^\tau \, (\iota_{_R}\ot \iota_{_S}) 
=    (\pi_{_R}\ot \pi_{_S})\,  (\iota_{_R}\ot \iota_{_S})
= 1_{C\ott D}
\ .$$

Now suppose in addition that $R$ and $S$ are graded and 
$C\sd$ and $D\sd$ are graded resolutions.
If the chain maps 
$\pi_{_R}$, $\pi_{_S}$, $\iota_{_R}$, $\iota_{_S}$ in each degree
are graded maps, 
then $\iota$ and $\pi$ in each degree are also
graded maps by construction.
\end{proof}

\vspace{2ex}

\subsection*{Special case: groups acting on polynomial rings}
Again take a finite group
       $G\subset \GL(V)$
       acting on $V=k^n$ inducing an action on 
        $k[V]=k[x_1, \ldots, x_n]$
        by automorphisms as in 
        \cref{GroupUnshuffle,GroupExampleAWEZmaps}.
 Here, $k[V]\ott kG=k[V]\# kG$
       for twisting map $\tau: g\ot f\mapsto {}^g\! f\ot g$
       for $g$ in $G$ and $f$ in $k[V]$.
We use the usual $G$-grading on the reduced bar complex
$\BB_{kG}$ with $g_1\ot\cdots\ot g_n$ homogeneous
of degree $g_1\cdots g_n$ for $g_1, \ldots, g_n$ in $G$. 
Consider  bimodule resolutions
$$
\begin{aligned}
C\sd\ :\ \ \cdots \longrightarrow & \ \, C_2 \longrightarrow C_1 \longrightarrow C_0 \longrightarrow 0
\ \ \text{ of } k[V], \text{ and}\\
D\sd\ :\ \ \cdots \longrightarrow & \ D_2 \longrightarrow D_1 \longrightarrow D_0 \longrightarrow 0
\ \ \text{ of } kG \, 
\end{aligned}
$$
for $D\sd$ $G$-graded with consistent left/right $G$-action
and $C\sd$ also a resolution of $k[V]$ as a left $kG$-module 
with consistent left actions of $G$ and $k[V]$.
So each $D_i=\bigoplus_{g\in G} (D_i)_g$ with
$$ g((D_i)_h)g'\subset (D_i)_{ghg'} \ \ \text{ and } \ \ 
g(f x)= (\, ^{g}\! f) (gx)
\ \ \text{ for }g,h,g'\in G, \, f\in k[V], \, x\in C\sd\, .
$$
Then $C\sd$ and $D\sd$ are both compatible with $\tau$ via bijective
chain maps
(see \cite{SW-grouptwistedAWEZ})
$$
\begin{aligned}
 &\tau_{_C}: \, kG\ot \, C\sd\ \ \longrightarrow\ \ C\sd\, \ot kG\, , 
  \qquad &&g\ot x &&\mapsto &&{}^gx\, \ot\ g 
  \quad\text{ for } g\in G,\ x\in C\, ,
  \\
  &\tau_{_D}: \, D\sd\, \ot k[V]\longrightarrow k[V]\ot D\sd\, , 
  \qquad &&y\ot f &&\mapsto &&{}^g\! f\, \ot \ y
  \quad\text{ for } f\in k[V],\ y\in D_g
    \, .
\end{aligned}
$$
The twisted product resolution $X\sd=C\sd\ott D\sd$
is a bimodule
resolution of $k[V] \# kG$
with bimodule action as follows: for 
  $x\in C_i$, $y\in (D_j)_h$, $g,g',h\in G$, and $f,f'\in k[V]$,
$$
(f'\ot g)(x\ot y)(f\ot g')
=
f' ( {}^gx) ( {}^{gh}\! f) \ot g yg'
\, .
$$

  Suppose there exist injective
  chain maps of bimodule resolutions of $k[V]$ and $kG$,
  respectively, compatible with $\tau$
   (with respect to $\tau_{_C}$ and $\tau_{_{D}}$),
$$
       \iota_{_{k[V]}}:C\sd\longrightarrow \BB_{k[V]}
      \qquad\text{ and }\qquad
       \iota_{_{kG}}:D\sd\longrightarrow \BB_{kG}\, .
       $$      
       For example, we might take the case when an injective chain map
       $\iota_{_{k[V]}}$
       is also
       a $kG$-module  
       map and an injective chain map $\iota_{_{kG}}$ is 
       $G$-graded.

If $kG$ is semisimple, i.e., 
char $k$ does not divide $|G|$,
then the map $\pi_{_{k[V]}}:  \BB_{_{k[V]}}\rightarrow C\sd$
that projects to $\Ima \iota_{_{k[V]}}$ with respect
to some $kG$-bimodule complement in $\BB_{k[V]}$ and then
applies $\iota_{_{k[V]}}^{-1}$ is a
$kG$-bimodule map and thus is
compatible with $\tau$.
If there is also a compatible chain map
$\pi_{_{kG}}: \BB_{kG}
\rightarrow 
D\sd 
$
with $\pi_{_{kG}}\, \iota_{_{kG}}=1$
(for example, 
take $\pi_{_{kG}}=1$ for $D\sd=\BB_{kG}$),
then \cref{barconversion} gives chain maps
$$
\pi:\BB_{k[V]\ot_{\tau} kG}   
\longrightarrow C\sd \ot _\tau D\sd
\quad\text{ and }\quad
\iota: C\sd \ot _\tau D\sd \longrightarrow 
\BB_{k[V]\ot_{\tau} kG}
$$
of  bimodule resolutions of $k[V]\# G$
with 
$\pi \, \iota 
= 1_{C \ot_{\tau} D}$.
Note that for homological considerations, one might take 
  $D\sd$ to give the augmented resolution 
  $0\rightarrow kG \rightarrow kG 
\rightarrow 0$.

If $kG$ is not
semisimple,
we may nevertheless be given chain maps 
$$
 \pi_{_{k[V]}}: \BB_{k[V]}\longrightarrow C\sd\, \quad 
             \qquad\text{and}\qquad 
              \pi_{_{kG}}:\BB_{kG}\longrightarrow D\sd\,\quad 
              $$
 with $\pi_{_{k[V]}}\, \iota_{_{k[V]}}=1_C$ and 
       $\pi_{_{kG}}\, \iota_{_{kG}}=1_D$
       with $\pi_{_{k[V]}}$ a $kG$-module map and $\pi_{_{kG}}$ 
       $G$-graded.
       Then $\pi_{_{k[V]}}$ and $\pi_{_{kG}}$
       are compatible with $\tau$
 and
\cref{barconversion} again gives chain maps
$$
\pi:\BB_{k[V]\ot_{\tau} kG}   
\longrightarrow C\sd \ot _\tau D\sd
\quad\text{ and }\quad
\iota: C\sd \ot _\tau D\sd \longrightarrow 
\BB_{k[V]\ot_{\tau} kG}
$$
of  bimodule resolutions of $k[V]\# G$
with 
$
\pi \, \iota 
= 1_{C \ot_{\tau} D}$.
We recover
\cite[Theorem 4.1]{SW-grouptwistedAWEZ},
where it was assumed implicitly that $\iota_{_{kG}}$ and $\pi_{_{kG}}$
were $G$-graded.
Explicit formulas for such chain maps are in~\cite[(2.6) and (4.2)]{SW-quantum}.

In the next sections, we develop results for the case when $kG$ is not
semisimple
and chain maps 
$
\pi_{_{k[V]}}: \BB_{k[V]}\rightarrow C\sd$
and
$\pi_{_{kG}}:\BB_{kG}\rightarrow D\sd$
       as in \cref{barconversion} are not given.
       See \cref{SymmetricAlgebraGroup,ChainMapsGroupActingOnPolyRing}.


\section{Conversion between Resolutions
using only compatible injective maps}
 Again consider $k$-algebras $R$ and $S$ and a twisting map 
$\tau: S\ot R\rightarrow R\ot S$ recording an  
algebra structure on $R\ot S$.
We use here the twisted
Alexander-Whitney and Eilenberg-Zilber maps constructed
in \cref{AW-EZforreducedbar}
to connect
the default reduced bar resolution $\BB_{R\ott S}$ for a twisted tensor
product to twisted product resolutions $C\ott D$ for that algebra,
but in the more difficult case when
compatible chain maps $\pi_{_R}$ and $\pi_{_S}$ as in the hypothesis of
\cref{barconversion}
may not be given.
We give a
splitting under more general conditions,

\vspace{3ex} 

\begin{equation*}
\xymatrix{
  &  C\sd\ot_\tau D\sd
  \ar[rr] &
  & \BB_{R\ot_{\tau} S}\ .
  \ar@(ul,ur)@{-->}[ll] \, 
}
\end{equation*}

\subsection*{Projective quotients}

We use a bootstrap construction guaranteed by
the following lemma
(see also~\cite[Lemma 4.7]{quad}).
%

\begin{lemma}\label{bootstrap}
Suppose $\psi: P\sd\rightarrow P'\sd$ is an injective chain map
of bimodule resolutions
of an algebra $A$ lifting the identity map on $A$.
If
$P'/\ima \psi $ is projective as an $A$-bimodule in each degree,
  then there exists a chain map  $\psi': P'\sd \rightarrow P\sd$
  with 
  $\psi'\, \psi=1_{_P}$:
  $$
\hspace{5ex}
\mbox{
\entrymodifiers={+!!<0pt,\fontdimen22\textfont2>}
\xymatrixcolsep{6ex}
\xymatrixrowsep{6ex}
\xymatrix{
P_{n}'
\ar[rr]^{\del_{n}'}   \ar[d]_{\psi'}
 &  & 
P_{n-1}'  \ar[d]_{\psi'}
&  & 
\\
P_{n}
 \ar[rr]_{\del_{n}}  \ar@<-1ex>[u]_{\psi}  
 & & 
\ P_{n-1}
\ar@<-1ex>[u]_{\psi} 
 & &
}
}
$$
If $A$ is a graded algebra with $P$, $P'$ graded resolutions
and $\psi$ a graded map, then $\psi'$ can be chosen to be graded.
\end{lemma}
\begin{proof}
  The result follows by induction using a diagram chase
  reminiscent of that verifying the Comparison Theorem as follows.
  Set $\psi'_{-1}=\psi_{-1}=1_A$ and let $\del$ and $\del'$
  denote the differentials on $P\sd$ and $P'\sd$, respectively,
  with $\del_0:P_0\rightarrow A$ and $\del'_0:P_0\rightarrow A$
  the augmentation maps.  Fix $n\geq 0$ and 
 assume that there exists a  chain map defined 
  up to  degree~$n-1$, i.e.,
  there are $A$-bimodule maps
$\psi_{i}': P_{i}'\rightarrow P_{i}$
commuting with the differentials with
$\psi_{i}'\, \psi_{i}=1_{P_i}$
for $0\leq i < n$.

We set $\psi_n'=\psi_n^{-1}$ on 
$\ima\psi_n$ in $P_n'$ (recall $\psi$ is injective) and notice that,
for $y$ in $\ima\psi_n$,
$$\del_n \, \psi_n'(y)
=
\psi_{n-1}'\, \psi_{n-1}\, \del_n \, \psi_n'(y)=
\psi_{n-1}'\, \del_{n}' \, \psi_n\, \psi_n' (y)
= \psi_{n-1}'\, \del_{n}' (y) . 
\, 
$$

Now we define $\psi_n'$ on a complement. 
Since the bimodule $P_n'/\Ima \psi_n$ is projective,
the map $P_n' \rightarrow P_n'/\Ima \psi_n$ splits,
and $$P_n'=E_n\oplus \Ima \psi_n$$
for some $A$-subbimodule $E_n$ of $P_n'$ 
with $E_n\cong P_n'/\Ima \psi_n$.
We define $\psi_n'$ on the direct summand $E_n$ as follows.
First note that
$$
     \psi'_{n-1}\, \partial'_n (E_n) \subseteq 
     \ima \partial_n
$$
since
$\ima \partial_n = \ker \partial_{n-1}$
and $\partial_{n-1} \psi'_{n-1} \partial_n' 
= \psi'_{n-2} \partial'_{n-1}\partial'_n = 0$
by induction.
Then since $E_n$ is projective and $\del_n: P_n\rightarrow
\ima \partial_n$ is onto,
the map $\psi'_{n-1}\, \partial'_n : E_n \rightarrow 
     \ima \partial_n$
lifts to a $A$-bimodule map 
$\psi_n ' : E_n\rightarrow P_n$ with
$$
   \partial_n\, \psi_n ' = \psi'_{n-1}\, \partial_n ' 
\qquad\text{ 
  as maps on $E_n$.}
$$

Then $\psi_n'$ commutes with the differential on all of $P_n'$, i.e.,
$\psi'$ is a chain map up to the degree~$n$ term, and
$\psi'_n \, \psi_n = 1_{P_n}$ by construction.
In the graded setting, one can choose $\psi'$ to be graded
(e.g., see~\cite[Corollary I.2.2]{NV82}). 
\end{proof}

We now use the twisted Eilenberg-Zilber map $\EZt$
of \cref{AW-EZforreducedbar}
for bimodule resolutions
\begin{equation*}
\begin{aligned}
  \quad&
C\sd: \ \ \cdots \longrightarrow C_2 \longrightarrow C_1 \longrightarrow C_0 \longrightarrow 0
\ \ \  \text{ of } R & \text{ and} 
\\
\quad&
D\sd:\ \  \cdots \longrightarrow D_2 \longrightarrow D_1 \longrightarrow D_0 \longrightarrow 0\
\ \text{ of } S & 
\end{aligned}
\end{equation*}
with compatible injective chain maps as shown:
 $$
\hspace{5ex}
\mbox{
\entrymodifiers={+!!<0pt,\fontdimen22\textfont2>}
\xymatrixcolsep{6ex}
\xymatrixrowsep{6ex}
\xymatrix{
(\BB_R )_{n+1}
\ar[rr]^{d_{n+1}}
 &  & 
 ( \BB_R )_{n}
&  & 
 \\
C_{n+1} 
 \ar[rr]^{\partial_{n+1}}  \ar@<-1ex>[u]^{\iota_{_R}}  
 & & 
\ C_{n} 
\ar@<-1ex>[u]_{\iota_{_R}} 
 & &
}
}
\hspace{-7ex}
\mbox{
\entrymodifiers={+!!<0pt,\fontdimen22\textfont2>}
\xymatrixcolsep{6ex}
\xymatrixrowsep{6ex}
\xymatrix{
(\BB_{S} )_{n+1}
\ar[rr]^{d_{n+1}}
 &  & 
(\BB_{S} )_{n}
&  & 
 \\
D_{n+1} 
 \ar[rr]^{\partial_{n+1}}  \ar@<-1ex>[u]^{\iota_{_S}}  
 & & 
\ D_{n}
\ar@<-1ex>[u]_{\iota_{_S}} 
 \ . &  &
}
}
$$

\vspace{1ex}

\begin{thm}\label{conversionspecial}
Suppose bimodule resolutions $C\sd$ and $D\sd$ of $k$-algebras $R$ and $S$,
  respectively, are compatible with a twisting map $\tau:S\ot
  R\rightarrow R\ot S$ and there exist  compatible
  injective chain maps 
  $$\iota_{_R}: C\sd\longrightarrow \BB_R\quad\text{ and }\quad
  \iota_{_S}: D\sd \longrightarrow \BB_S\, .$$
If $\BB_{R\ott S}/\Ima \iota$ is
  projective
  as an $\ttpp$-bimodule
  in each degree for $\iota=\EZt(\iota_{_R}\ot \iota_{_S})$,
  then there exists a chain map 
$
\pi:\BB_{R\ot_{\tau} S}   
\longrightarrow C\sd \ot _\tau D\sd$
of bimodule resolutions
of $\ttp$ with composition the identity:
$$
\pi \, \iota 
= 1_{C \ot_{\tau} D}\, .
$$
If $R$ and $S$ are graded algebras with $C\sd$ and $D\sd$ graded
resolutions, then the chain maps $\iota$, $\pi$ are graded maps provided 
$\iota_{_R}$ and $\iota_{_S}$ are graded maps.
 $$
\entrymodifiers={+!!<0pt,\fontdimen22\textfont2>}
\xymatrixcolsep{2ex}
\xymatrixrowsep{8ex}
\xymatrix{
(\BB_{R \,\ott S})_{n+1}   
\ar[rr]  \ar@<-1ex>[d]_{\pi}
&  & 
 (\BB_{ R\,\ott S})_{n}
 \ar@<-1ex>[d]_{\pi}
\\
(C\sd \otimes_{\tau} D\sd )_{n+1} 
 \ar[rr]  \ar[u]<-1ex>_{\iota}
 & & 
(C\sd \otimes_{\tau} D\sd )_{n} \, 
\ar[u]<-1ex>_{\iota}
}
$$
\end{thm}
\begin{proof}
  We apply \cref{bootstrap} with $A=\ttp$ and $\psi=\iota$ noting that
$\EZt$ (see \cref{AW-EZforreducedbar}) is injective
as is $\iota_{_R}\ot \iota_{_S}$ since it is a tensor product
of injective maps  over a field.
 \end{proof}

 \subsection*{Injective chain maps preserving spanning sets}

 \cref{conversionspecial}
gives the following conversion maps.

\begin{thm}\label{conversionfree}
Suppose
bimodule resolutions $C\sd$ and $D\sd$ of $k$-algebras $R$ and $S$,
respectively, are compatible with
the twisting map $\tau$.
 Suppose there exist  
 injective compatible chain maps 
  $$\iota_{_R}: C\sd\longrightarrow \BB_R\quad\text{ and }\quad
  \iota_{_S}: D\sd \longrightarrow \BB_S$$
  with
  $\Ima \iota_{_R}$ and   $\Ima \iota_{_S}$
  spanned by their intersections with
  $k\ot \bar R^{\ot n}\ot k$ and $k\ot \bar S^{\ot n}\ot 
  k$ as bimodules over $R$ and $S$, respectively,
  for each $n$.
 Then there exist chain maps 
$$
\pi:\BB_{R\ot_{\tau} S}   
\longrightarrow C\sd \ot _\tau D\sd 
\quad\text{ and }\quad
\iota: C\sd \ot _\tau D\sd \longrightarrow 
\BB_{R\ot_{\tau} S}
\, 
$$ 
of bimodule resolutions
of $\ttp$ with composition the identity:
$$
\pi \ \iota 
= 1_{C \ot_{\tau} D}\, .
$$
If $R$ and $S$ are graded algebras with $C\sd$ and $D\sd$ graded
resolutions, then the chain maps $\iota$, $\pi$ are graded maps provided 
$\iota_{_R}$ and $\iota_{_S}$ are graded maps.
\end{thm}
\begin{proof}
  Let $A=\ttp$ and again note that the $A$-bimodule map 
  $\iota=\EZt(\iota_{_R}\ot \iota_{_S})$ is injective (see the proof
  of  \cref{conversionspecial}).
  For each $n\geq 0$, set
  $$(\BB_R')_n=k\ot \bar R^{\ot n}\ot k\ \subset (\BB_R)_n
  \quad\text{ and }\quad
  (\BB_S')_n=k\ot \bar S^{\ot n}\ot k\ \subset (\BB_S)_n
  \, .
  $$
  Let $C'_n$ and $D'_n$ be sets spanning $C_n$ and $D_n$ as bimodules 
  over $R$ and $S$, respectively, with
  $$
  \iota_{_R}(C_i')\subset \BB_R'
\quad\text{ and }\quad
\iota_{_S}(D_i')\subset \BB_S'
\, .
  $$
  
  As twisting preserves the identities of $R$ and $S$,
  the compatibility maps $\tau_{_{\BB_R}}$ and $\tau_{_{\BB_S}}$
  (see   \cref{BarCompatible}) preserve
  $\BB_R'$ and $\BB_S'$.
  Thus
  the compatibility of $\iota_{_R}$ and $\iota_{_S}$
  give a commutative diagram
  for each $i,j\geq 0$,
$$
  \begin{tikzcd}[row sep=5ex, column sep=17ex, ampersand replacement=\&]
    R\ot C_i'\ot R\ot S\ot D_j'\ot S
    \arrow{r}{1\ot \iota_{_R}\ot 1\ot 1\ot \iota_{_S}\ot 1}
    \arrow[swap]{d}{1\ot 1\ot \tau\ot 1 \ot 1}
    \&
    R\ot (\BB_R')_i \ot R \ot S\ot (\BB_S')_j\ot S
     \arrow{d}{1\ot 1\ot \tau\ot 1\ot 1}
   \\
   R\ot C_i'\ot S\ot R\ot D_j' \ot S
   \arrow{r}{1\ot \iota_{_R}\ot 1\ot 1\ot \iota_{_S}\ot 1}
    \arrow[swap]{d}{1\ot \tau_{_{C}}\ot \tau_{_{D}}\ot 1 }
   \&
   R\ot (\BB_R')_i \ot S \ot R\ot (\BB_S')_j\ot S
    \arrow[]{d}{1\ot \tau_{_{\BB_R}}\ot \tau_{_{\BB_S}}\ot 1 }
    \\
    R\ot S\ot C_i\ot D_j\ot R\ot S
    \arrow{r}{1\ot 1\ot \iota_{_R}\ot \iota_{_S}\ot 1\ot 1}
    \arrow{d}
    \&
    R\ot S\ot (\BB_R')_i\ot (\BB_S')_j\ot R\ot S
    \arrow{d}
    \\
    C_i\ot D_j
    \arrow{r}{\iota_{_R}\ot\iota_{_S}}
    \&
 (\BB_R)_i \ot  (\BB_S)_j 
     \end{tikzcd}
$$
where the unlabeled arrows indicate the bimodule
action of $R\ot S$ 
after identifying it with $A$ as a vector space.

Any element of $C_i\ot D_j$ lies in the image
of the composition of the maps on the left side of the diagram
since $\tau_{_C}$ and $\tau_{_D}$ are compatibility maps,
and hence its image under $\iota_{_R}\ot \iota_{_S}$
lies in the image of the maps on the right side,
i.e., in the 
$A$-bimodule spanned by $(\BB'_R)_i\ot (\BB_S')_j$.
But the chain map $\EZt$ takes the $k$-subspace
$$
(\BB'_R)_i\ot (\BB_S')_j
\quad\text{ of }\quad
(\BB_R)_i\ot (\BB_S)_j
$$
into
the $k$-subspace $$k\ot \bar A^{\ot (i+j)}\ot k
\quad\text{ of }\quad
(\BB_A)_{i+j}$$ by construction.
Hence
in each degree $n$, the $A$-bimodule map
$$\iota_n: (C\ott D)_n \longrightarrow (\BB_A)_n$$
 has image spanned as an $A$-bimodule by a $k$-vector subspace $U$
of $k\ot \bar A^{\ot n}\ot k$.

We choose a
$k$-vector space complement $U^c$ to $U$ in
$k\ot \bar A^{\ot n}\ot k$.  Then a $k$-vector space basis
for $U^c$
gives a free $A$-bimodule basis for $\BB_A/\Ima \iota$
by construction.
The  result then follows from \cref{conversionspecial}.

If $R$ and $S$ are graded, $C\sd$ and $D\sd$ are graded resolutions,
and the chain maps 
$\iota_{_R}$, $\iota_{_S}$ in each degree
define graded maps,
then $\iota$ and $\pi$ in each degree are also
graded maps by construction.
\end{proof}


\section{Applications: Smash products}
\label{sec:smashproducts}

We now specialize to smash product algebras arising from 
actions of groups and Hopf algebras.
In the next section, we further specialize to actions on Koszul algebras
which will be used to explore deformations in~\cite{SW-DeformationTheoryHopfActions}.
Specifically, there we will use 
conversion chain maps established here to describe PBW deformations of
smash product algebras $R\# H$
arising from Hopf algebras acting on Koszul algebras,
giving
Poincar\'e-Birkhoff-Witt conditions 
for Hopf-Koszul Hecke algebras.
Note that in comparison to~\cite{quad} and~\cite{SW-grouptwistedAWEZ},
we have chosen to write elements of the Hopf algebra (or group algebra)
on the right so that a smash product $R\# H$ is identified with a
twisted
tensor product $R\ott H$.

\subsection*{Hopf algebra}
Let $\Ho$ be a Hopf algebra over the field $k$ 
with bijective antipode $\gamma$, 
coproduct $\Delta$, 
and counit $\varepsilon$.
We use standard Sweedler notation to write
the coproduct $\Delta$  symbolically:
$\Delta(h)=\sum h_1\ot h_2$.
See \cite{Radford} or \cite{Montgomery1993}
for basic notions. 

\subsection*{Hopf algebra actions}
Let $R$ be a $k$-algebra carrying a Hopf action of $\Ho$, i.e.,
$R$ is a (left) Hopf module algebra over $\Ho$. 
We denote the action of $\Ho$ on $R$ by left superscript,
$r\mapsto\ {}^{h} r$.  
Thus $R$ is a $k$-algebra which is a left $\Ho$-module
with
$\, ^{h}(rr')= \sum \, ^{h_1}(r)\, ^{h_2}(r')$
and 
$\ ^{h}1_R = \epsilon(h) 1_R$
for $h$ in $\Ho$ and $r,r'$ in $R$.
We use the induced {\em right action} of $\Ho$ on $R$ as well given by
$r \cdot h = {}^{\igamma(h)}r$.
When $R$ is graded, we assume the action of $\Ho$ preserves
the grading, i.e.,
$\deg( ^{h}r)= \deg(r)$ for homogeneous $r$ in $R$.

\subsection*{Smash product algebra}
The smash product algebra
$R\# \Ho$ is the free $R$-module with free basis given by
a vector space basis of $\Ho$ 
and multiplication
\[
   (rh) \cdot (r'h') = \sum r ({}^{h_1}r')\, h_2 h'
   \quad\text{for all}\quad r,r' \in R\ \text{ and }\ h, h'\in H . 
\]
We may realize the smash product algebra as a twisted tensor product:
$$ R\# \Ho \cong R \ott H $$
for twisting map
$$
{}_{}\hspace{-2ex}
\tau: H\ot R  \longrightarrow R\ot H
\quad\text{ defined by }\quad 
h\ot r \mapsto \ \sum { ^{h_1} r} \ot h_2 \, 
$$
and inverse twisting map 
$$
\tau^{-1}: R\ot H \longrightarrow H\ot R
\quad\text{ defined by }\quad 
r\ot h \mapsto \sum h_2\ot \, ^{\igamma(h_1)} r
$$
for $h\in \Ho$ and $r\in R$.

To see that $\tau$ is invertible with inverse as indicated, note that
for all  $h$ in $H$,
$$\varepsilon(h) = \igamma\big(\varepsilon(h)\big)
= \igamma\big( \sum \gamma(h_1) h_2\big)
= \sum \igamma(h_2)\, h_1 
$$
as  $\igamma$ is an algebra antimorphism, and
thus
$$\sum h_3\ot \igamma(h_2)\, h_1
=\sum h_2\ot \epsilon(h_1)
=\sum \epsilon(h_1) \, h_2\ot 1_H
=h\ot 1_H
\,, $$
using Sweedler notation
$(1_\Ho\ot \Delta)\Delta(h)=(\Delta\ot 1_\Ho)\Delta(h)=\sum h_1\ot
h_2\ot h_3$.  
This then implies that for $r$ in $R$
$$\itau\, \tau(h\ot r) 
  =
  \sum h_{3} \ot \, ^{\igamma(h_{2})h_1} r
  =
  \sum h_{2} \ot \, ^{\epsilon(h_1)1_H} r
  =
    \sum h_{2} \epsilon(h_1)   \ot r
    =h\ot r\,.
 $$
An analogous argument using $\epsilon(h)
=\sum h_2\, \igamma(h_1)$
confirms that $\itau$ thus defined is right inverse to $\tau$.


\subsection*{Resolutions carrying the Hopf action}\label{subsec:tpr}
We establish some terminology for resolutions 
of the algebra $R$
that respect the Hopf action and resolutions of 
the Hopf 
algebra $\Ho$ which are also Hopf modules.

\vspace{2ex}

\begin{definition}\label{carries}{\em 
We say that an $R$-bimodule resolution $C\sd$ of $R$,
\begin{equation*}
\begin{aligned}
& C\sd\ :\ \  \cdots \longrightarrow C_2 \longrightarrow C_1 \longrightarrow C_0
\longrightarrow 0\ ,
\end{aligned}
\end{equation*}
{\em carries the Hopf action of $\Ho$}
when each $C_{i}$ is a left $H$-module,
$c\mapsto \, ^hc$, with
\begin{equation}\label{eqn:asxs}
  \, ^h(r c \,r') = \sum   {}^{h_1}r \
                        {}^{h_2}c\
                        {}^{h_3}r' 
                        \quad\text{for $h$ in $\Ho$, $r,r'$ in $R$,
                          and $c$ in $C_i$}
\end{equation}
and differential $\del^{}_C$ and augmentation map $C_0\rightarrow R$
both  left $\Ho$-module homomorphisms.
}
\end{definition}

\vspace{2ex}

For any $H$-bimodule resolution $D\sd$ of $H$,
we view each $H\ot D_i$ as an $H$-bimodule
via the diagonal action, i.e., 
\begin{equation}\label{actionHD}
  h(a\ot y)h'=\sum   h_1a \, h'_1\ot   h_2\, y \, {h'_2}
\qquad\text{ for $h,a, h'\in \Ho$ and $y \in D_i$}
\, . 
\end{equation}

\vspace{2ex}

\begin{definition}\label{HopfResolution}{\em 
  A {\em left comodule bimodule resolution} of the Hopf 
  algebra $\Ho$
is an $H$-bimodule resolution $D\sd$ of $\Ho$, 
\begin{equation*}
\begin{aligned}
D\sd\ : \ \ \cdots \longrightarrow D_2 \longrightarrow D_1 \longrightarrow D_0 
\longrightarrow 0\ , 
\end{aligned}
\end{equation*}
with 
each $D_i$ a left $\Ho$-comodule 
via comodule maps that are $\Ho$-bimodule 
homomorphisms and with
differential $\del^{}_D$ and
augmentation $D_0\rightarrow \Ho$ 
both  $\Ho$-comodule 
homomorphisms. 
}
\end{definition}

\vspace{2ex}

\begin{remark}\label{eqn:aca}
  {\em
    We denote    the left $\Ho$-comodule map on
    a left comodule bimodule resolution $D\sd$
    of $\Ho$
    by $\rho: D\sd\rightarrow \Ho\ot D\sd$
    written using Sweedler notation as
    $y\mapsto \sum y_1\ot y_2$. Then
    as $\rho$ is an $\Ho$-bimodule 
    homomorphism (see \cref{actionHD}),
\begin{equation*}
   \rho ( h y h') = \sum h_1\, y_{1}\, h'_1 \ot h_2\, y_2\, h'_2
   \quad\text{ for } y\in D_i, h,h'\in \Ho, \text{ and all degrees } i
   \, .
 \end{equation*}
}
\end{remark}

\vspace{2ex}

We now observe that
resolutions respecting the action of a Hopf algebra $\Ho$
are compatible with the twisting map $\tau$ recording the action.

\begin{lemma}\label{compatibleHopfresolutions}
Let $\Ho$ be a Hopf algebra acting on an algebra $R$. 
 Any bimodule resolution $C\sd$ of $R$ carrying the Hopf action of $\Ho$
   is compatible with the twisting map $\tau:H\ot R\rightarrow R\ot H$ via a chain map  
$\tau^{}_{_C}:H\ot C\sd\rightarrow C\sd\ot H$ 
given by 
$$
    \tau^{}_{_C}(h\ot x) = \sum  {}^{h_1} x\ot h_2  
    \quad\text{for all $h\in \Ho$, $x\in C\sd\, $.}
$$
Any left comodule bimodule resolution $D\sd$ of $\Ho$ is compatible with
$\tau$ via a chain map $\tau^{}_{_D}: D\sd\ot R \rightarrow R\ot D\sd$ given by 
$$
   \tau^{}_{_D}(y\ot r) = \sum {}^{y_1}r\ot y_2 
   \quad\text{ for all $y\in D\sd$, $r\in R$}\, .
   $$
\end{lemma}
\begin{proof}
   The map  $\tau^{}_{_C}$ is a chain map
 lifting the augmentation map from $C\sd$ to $R$ 
since the differential $\del^{}_C$ is  
an $\Ho$-module homomorphism. 
Commutativity with the multiplicative structure of $R$ (see
\cref{CompatibleResolution}) holds since the
comultiplication $\Delta$ is an algebra homomorphism, and
commutativity with the bimodule structure holds by \cref{eqn:asxs}.
Hence  $C\sd$ is compatible with $\tau$.
One can check that the map  $\tau^{}_{_D}$
defines a chain map lifting the comultiplication $\Delta$ of $\Ho$
since the differential $\del^{}_D$ on $D\sd$
is an $H$-bimodule map and also an
$\Ho$-comodule map.
Commutativity with the bimodule structure of $D\sd$ and with multiplication 
can be verified using \cref{eqn:aca}.
Hence
$D\sd$ is also compatible with $\tau$.
\end{proof}

\vspace{2ex}

\begin{example}\label{barcarriesHopfaction}
  {\em
The bar resolution $\B_R$ and reduced bar resolution $\BB_R$ 
carry the Hopf action of $\Ho$ 
with left $\Ho$-module action
given as usual on $\B_R$ via the coproduct $\Delta$ of $\Ho$,
$$
\, ^{h} (r_1\ot\cdots\ot r_{n})
=\sum
\, ^{ h_1}r_1\ot\cdots\ot
\ ^{ h_n}r_n
\quad\text{for $r_i$ in $R$ and $h$ in  $\Ho$} .
$$
This induces a left $H$-module action on $\BB_R$
since $ {}^h1_R = \varepsilon(h)1_R$ for $h$ in $H$.
\cref{BarCompatible} and \cref{compatibleHopfresolutions}
give the same compatibility maps
because the antipode of a Hopf algebra 
and its inverse are coalgebra antimorphisms:
Explicitly,
      $$\tau^{}_{_{\BB_R}}:H\ot \BB_R\longrightarrow \BB_R\ot H,
    \quad 
    h\ot c \mapsto 
   \sum  {}^{h_1} c\ot h_2  
   \quad\text{for $h\in \Ho$, $c\in \BB_R $}
   \, .
   $$
   }
\end{example}

\vspace{2ex}

\begin{example}\label{BarIsComoduleRes}
  {\em
The bar resolution $\B_\Ho$ 
and reduced bar resolution  $\BB_\Ho$ of $\Ho$
are both left comodule bimodule resolutions of $\Ho$
with the $\Ho$-comodule map $\rho$ given on $\B_{\Ho}$ by
$$
   \rho ( h^0\ot h^1 \ot\cdots \ot h^{n+1}) 
   =  (h^0_1 \cdots h^n_1 h^{n+1}_1)\ot 
     h^0_2\ot \cdots  \ot h^{n+1}_2 
     \quad\text{for $h^i\in \Ho$}
     \, 
     $$
     and induced (using projection) on $\BB_{\Ho}$ using
     the fact that $\Delta(1_H)=1_H\ot 1_H$.
     \cref{BarCompatible} and \cref{compatibleHopfresolutions}
give the same compatibility maps:
Explicitly,
  \[
    \begin{aligned}\tau^{}_{_{\BB_H}}
      : \BB_H\ot R &\longrightarrow R\ot \BB_H, \\
  \
  (h^0\ot\cdots\ot h^{n+1})\ot r &\longmapsto
\sum \ ^{ (h^{0}_{1}\cdots h^{n+1}_1)}r \ot (h^0_2\ot \cdots\ot h^{n+1}_2), 
\end{aligned}
\]
for $h^0, \ldots, h^{n+1}$ in $\bar H\subset H$,
a sum for $n+2$ indices using Sweedler notation. 
(For brevity, we have suppressed notation for a 
projection map $\pr^{}_{\BB_H}:\B_H\rightarrow \BB_H$.)
See \cref{SymmetricAlgebraGroup}
for the special case where $H=kG$ for a finite group $G$.
}
\end{example}

\vspace{2ex}

We will see in the next section
that Koszul resolutions also carry the Hopf action in the graded case.

\subsection*{Twisted product resolution for Hopf actions}
We form the twisted product resolution from the resolutions
$C\sd$ and $D\sd$.
Assume $C\sd$ is a resolution of $R$ carrying the Hopf action of $\Ho$
and
$D\sd$ is a  left comodule bimodule resolution
of $\Ho$.
By \cref{compatibleHopfresolutions}, $C\sd$ and $D\sd$
are compatible with the twisting map $\tau:H\ot R\rightarrow R\ot H$,
and we form
the twisted product resolution $C\sd \ott D\sd$
as 
the total complex of the double complex 
$C\sd\ot D\sd$ (see
(\ref{cx-X}) and
(\ref{twistedproductresolution}))
with $(R\# \Ho)$-bimodule structure on each $C_i \ot D_j$ 
given by Diagram~(\ref{bimodstructure}). 
Explicitly,
the left action of $H$ and the right action of 
$R$ are given by
\[
   h \cdot (x\ot y)\cdot r
   \ = \
   \sum \, ( ^{h_1} x)\  ( ^{h_2 y_1} r) \ot h_3\, y_2\
\quad\text{ for }\ \ 
h\in \Ho,\ x\in C\sd,\ y\in D\sd,\ r\in R\, . 
\]
Then the total complex of $C\sd \ott D\sd$, augmented by $R\# \Ho$,
is indeed an exact sequence of $(R\# \Ho)$-bimodules,
see~\cite[Lemma~3.5]{SW-twisted}.
It
is often a projective resolution of $R\# \Ho$ in case
$C\sd$ is a projective resolution of $R$
and $D\sd$ is a projective resolution of $H$,
see~\cite[Theorem~3.10]{SW-twisted}.  

\subsection*{Conversion between resolutions
for smash products}\label{sec:conversion-2}
One wishes to navigate 
between the reduced bar resolution $\BB_{R\# H}$
and various twisted product resolutions $C\sd\ott D\sd$
of $R\# \Ho$.
We consider
compatible chain maps
$\iota_{_R} : C\sd\rightarrow \BB_R$ and
$\iota_{_H}: D\sd\rightarrow \BB_\Ho$:
$$
\hspace{5ex}
\mbox{
\entrymodifiers={+!!<0pt,\fontdimen22\textfont2>}
\xymatrixcolsep{6ex}
\xymatrixrowsep{6ex}
\xymatrix{
(\BB_R )_{n+1}
\ar[rr]^{d_{n+1}}   
 &  & 
( \BB_R)_{n}   
&  & 
 \\
C_{n+1} 
 \ar[rr]^{d_{n+1}}  \ar@<-1ex>[u]_{\iota_{_R}}  
 & & 
\ C_{n} 
\ar@<-1ex>[u]_{\iota_{_R} } 
 & &
}
}
\hspace{-7ex}
\mbox{
\entrymodifiers={+!!<0pt,\fontdimen22\textfont2>}
\xymatrixcolsep{6ex}
\xymatrixrowsep{6ex}
\xymatrix{
(\BB_{\Ho} )_{n+1}
\ar[rr]^{d_{n+1}}  
 &  & 
(\BB_{\Ho} )_{n} 
&  & 
 \\
D_{n+1} 
 \ar[rr]^{d_{n+1}}  \ar@<-1ex>[u]_{\iota_{_\Ho}}  
 & & 
\ D_{n}
\ar@<-1ex>[u]_{\iota_{_\Ho}} 
\ . &  &
}
}
$$

\vspace{2ex}

The following is a corollary of \cref{conversionfree},
using \cref{compatibleHopfresolutions}.
\begin{cor}\label{hopfconversion} 
  Let $\Ho$ be a Hopf algebra
  acting on an algebra $R$.
Consider any 
twisted product resolution $C\sd \ott D\sd$ of
$R \# \Ho$ twisting a bimodule resolution 
$C\sd$ of $R$ carrying the Hopf action of $\Ho$
with a left comodule
bimodule resolution
$D\sd$ of $H$.
Suppose there exist injective chain maps compatible with $\tau$
$$
\iota_{_R}: C \sd \longrightarrow \BB_R
\qquad\text{ and } \qquad
  \iota_{_H}: D\sd\longrightarrow \BB_H \quad
  $$
  with   $\Ima \iota_{_R}$ and   $\Ima \iota_{_H}$ spanned by 
  their intersections with 
  $k\ot \bar R^{\ot n}\ot k$ and  
    $k\ot \bar H^{\ot n}\ot k$ 
    as bimodules over $R$ and $H$, respectively, for each $n$.
    Then there exist chain maps of bimodule resolutions of $R\# H$
$$\iota: C\sd \ott D\sd\longrightarrow 
\BB_{R\# \Ho}
\quad\text{ and }\quad
\pi:\BB_{R\#\Ho}
\longrightarrow C\sd \ott D\sd
$$
with
$$\pi \ \iota
= 1_{C \ot_{\tau} D}\, .
$$
If $R$ is a graded algebra and $\Ho$ is graded concentrated in degree~0,
then the chain maps $\iota$ and $\pi$ are graded maps provided 
$\iota_{_R}$ and $\iota_{_H}$ are graded maps.
\end{cor} 


\section{Applications to Hopf actions on Koszul algebras}
\label{sec:Hopf-Koszul}

We now specialize to a Hopf algebra $\Ho$ acting on a quadratic
algebra $R$ that is Koszul, recording the theory needed to 
simultaneously generalize some of the results of 
\cite{Khare}, \cite{SW-Koszul}, and \cite{WW} in 
our sequel~\cite{SW-DeformationTheoryHopfActions}.
All our Koszul algebras will be assumed to be
finitely generated connected graded
algebras, so we take $R$ to be the quotient of a tensor
algebra of a finite dimensional vector space in degree~$1$.
One might identify $\bar{\Ho}=\Ho/k1_\Ho$ with the vector subspace
$\ker(\varepsilon)$ of $H$ (for $\varepsilon$ the counit),
choosing a corresponding section to the projection map $H\twoheadrightarrow
H/k1_H$.


\subsection*{Koszul algebra and Koszul resolution}
Suppose $R = T_k(V)/(\R)$ is generated by a finite
dimensional $k$-vector space $V$ with
{\em quadratic Koszul relations} given as some subspace
$\R\subset V\ot V$,
where $T_k(V)$ is the tensor algebra of $V$ over $k$.
We take the grading on  $T_k(V)$ with $V$ in degree $1$ and identify $V$ with the vector subspace $(V+(\R))/(\R)$ of $R$ in 
degree~1.  We also identify $\bar R=R/k1_R$ with the $k$-span of the positively 
graded 
elements of $R$ so that $V\subset \bar R$.
We will use the {\em Koszul resolution} $(K_R)\sd=K\sd$ for $R$ 
\begin{equation}
  \label{koszulres}
  K\sd\ :\ \  \cdots\longrightarrow K_3\stackrel{d_3}{\longrightarrow}
  K_2 \stackrel{d_2}{\longrightarrow} 
  K_1 \stackrel{d_1}{\longrightarrow}
  K_0 \longrightarrow R\longrightarrow 0
\end{equation}
where $K_0 = R\ot R$, $K_1=R\ot V\ot R$,
$K_2=R\ot \R \ot R$, and, more generally,
\begin{equation}\label{Koszulterms}
  K_n = R\ot \widetilde K_n
    \ot R 
    \quad\text{ for } n\geq 2\, 
    \quad\text{ where } \quad
    \widetilde K_n=\bigcap _{j=0}^{n-2} (V^{\ot j}\ot \R 
  \ot V^{\ot (n-2-j)} ) 
\end{equation}
with differential given by that for the reduced bar 
resolution under the standard inclusion chain map,
$\iota_{_R}: 
K_R\longrightarrow \BB_R$,
defined in degree $n$ by the canonical inclusion of $K_n$
into $R\ot {\bar R}^{\ot n}\ot R$.

\subsection*{Koszul algebra carrying an action of a Hopf algebra}
Now assume that the action of the Hopf algebra $H$ on the Koszul algebra $R=T(V)/\R$
is graded, i.e., 
$\deg(\, ^h r) = \deg(r)$ for $r$  in $R$ homogeneous
and $h$ in $\Ho$.
We put the trivial grading on $\Ho$, i.e.,
we set $\deg h=0$ for all $h\in \Ho$.
We use the induced right action of $\Ho$ on $R$,
$r\mapsto r\cdot h=\, ^{\igamma(h)}r$ for $h$ in $\Ho$
and thus regard $R$ as a right $H$-module as well when needed.

The action of $\Ho$ on $R$ restricts to an action  
of $\Ho$ on $V$ since  it preserves the grading.
This induces an action on $V\ot V$,
$$\, ^h(v\ot w)=\sum \, ^{h_1}v\ot \, ^{h_2} w
\quad\text{ for $h$ in $\Ho$ and $v,w$ in $V$,}
$$
congruent with the action on $R$:
$\, ^h(v\ot w + (\R)) 
= \, ^h(v\ot w)+(\R) 
$ for $v$, $w$ in $V$.
Thus
the Hopf algebra $\Ho$ 
must preserve the set of quadratic Koszul relations
$\R\subset V\ot V$
defining the Koszul algebra $R=T(V)/(\R)$,
\begin{equation}
   \label{hopfpreserveskoszulrelations}
  ^h \R
\subset\R\quad \text{ for all }
h\in \Ho\, .
\end{equation}

By \cref{hopfpreserveskoszulrelations}, 
each term $K_n$ of the Koszul resolution 
of $R$ is a left $H$-module 
with action denoted $c\mapsto \, ^hc$
induced from the action of $H$ on $R\ot V^{\ot n}\ot R\supset K_n$
given by 
\begin{equation}\label{HopfActionKoszulResolution}
  ^h(r_1\ot v_2\ot \cdots \ot v_{n+1}\ot r_{n+2})=
\sum \  ^{h_1}r_1\ot\, ^{h_2}v_2\ot\cdots \ot \, ^{h_{n+1}}v_{n+1}\ot 
\, ^{h_{n+2}}r_{n+2}
\, . 
\end{equation}
The Koszul resolution 
$K\sd$ of $R$ then 
carries the action of $H$ (see Definition~\ref{carries}): 
$$
\, ^{h} (r\, c\, r')=\sum\ \, ^{h_1}r\ ^{h_2}c\ ^{h_3} r' 
\text{
  for $c$ in $K_n$, $r$, $r'$ in $R$, and $h$ in $H$.}
$$
Consequently, 
\cref{compatibleHopfresolutions} implies $K\sd$
   is compatible with $\tau$ via compatibility map
   \begin{equation}\label{KoszulIsCompatible}
     \tau^{}_{_C}:H\ot C\sd\longrightarrow C\sd\ot H,
   \quad 
   h\ot c \mapsto 
   \sum  {}^{h_1} c\ot h_2  
   \quad\text{for $h\in \Ho$, $c\in C\sd\, $}
   \, .
 \end{equation}
 Alternatively, see~\cite[Proposition 2.20(ii)]{SW-twisted} noting
 that
 $\tau$ is strongly graded in this setting.

 \vspace{2ex}
 
\begin{example}\label{SymmetricAlgebraGroup}
  {\em 
Again consider $k[V]$, the polynomial ring on $V\cong k^n$, with a linear action of a finite group
$G$
as in \cref{GroupUnshuffle,GroupExampleAWEZmaps}
and the  end of~\cref{sec:conversion}.
Let $C\sd=K\sd$ be the (bimodule) Koszul resolution of $k[V]$,
$$
  K\sd:\  \cdots \longrightarrow 
k[V]\ot \Wedge^2 V \ot k[V]
  \longrightarrow k[V]\ot V\ot k[V] \longrightarrow k[V]\ot k[V]
  \longrightarrow 0\, ,
$$
and let $D\sd=\BB_{kG}$ be the reduced bar resolution of the Hopf algebra
$\Ho=kG$.
Then $C\sd$ is a free resolution of $k[V]$ carrying the action of $kG$ and
$D\sd$ is a free left comodule bimodule resolution with $\Ho$-comodule map
$\B_{kG}\rightarrow kG\ot \B_{kG}$ given by
$$ g_0\ot\cdots \ot g_{n+1}\longmapsto (g_0\cdots g_{n+1} )\ot g_0\ot 
\cdots\ot g_{n+1}\, \quad\text{ for } g_i \in G\, .
$$
By \cref{compatibleHopfresolutions}
(also see \cref{BarCompatible}),
$C\sd$ and $D\sd$ are compatible
with $\tau$.  
}\end{example}

\vspace{2ex}

\subsection*{Twisted product resolutions for Hopf algebras acting on Koszul algebras}

Let
$X\sd=K_R\ott \BB_H$ (see \cref{cx-X})
be the twisted product resolution
combining the Koszul resolution $K_R=K\sd$ of $R$ 
defined in~(\ref{koszulres})
with the reduced
bar resolution $\BB_\Ho$ of $\Ho$.
Specifically,  
\begin{equation}\label{eqn:xij}
  X_n=\bigoplus_{i+j=n}X_{i,j}
  \qquad\text{ for } X_{i,j} = (K_R)_i\ot (\Ho\ot\bar{\Ho}^{\ot j}\ot \Ho) 
  \qquad i,j\geq 0\, .  
\end{equation} 
Note here that $X\sd$ is a {\em free} $R\# \Ho$-bimodule resolution
of $R\# \Ho=R\ott H$. Indeed,
$K_R$ is a free $R^e$-module in each degree and $\BB_H$ is a free
$\Ho ^e$-module in each degree, and it can be shown directly that $X_{i,j}$
is a free $(R\# \Ho)$-bimodule 
with basis given by tensoring free basis elements of $(K_R)_i$ with those
of $(\BB_H)_j$.
Specifically, the twisting map and its inverse
(see~Section~\ref{sec:smashproducts})   
may be used
to provide an isomorphism of vector spaces
\[
 (K_R\ot \BB_H)_n \ \cong \  \bigoplus_{\ell + i =n}
   (R\# H) \ot \left( \cap_{j=0}^{\ell -2}  (V^{\ot j}\ot \R\ot V^{\ot (\ell -2 -j)}
   \right) \ot {\bar{H}}^{\ot i}\ot (R\# H)
\]
that may be checked to be 
compatible with the $(R\# H)$-bimodule structures. 
The claimed free basis
can be seen in the latter expression.
See \cite[Corollary 3.12]{SW-twisted}.
 

%


\subsection*{Chain maps for Hopf algebras acting on Koszul algebras}

In~\cite{SW-DeformationTheoryHopfActions},
we will use the Alexander-Whitney and Eilenberg-Zilber maps
of Section~\ref{sec:AW-EZ} to describe all filtered PBW deformations
of $R\# H$ and
generalize
\cite[Theorem 2.5]{SW-Koszul}.
Results depend heavily on the next theorem 
on
the twisted product resolution
$X\sd=K_R\ott \BB_H$.
 \begin{thm}\label{LastThm}
   Let $X\sd=K_R\ott \BB_H$
be the twisted product resolution for 
$R\# \Ho$ that twists the Koszul resolution $K_R$ of $R$
with the reduced bar resolution 
$\BB_\Ho$ of the Hopf algebra $\Ho$. 
There are 
graded chain maps of bimodule resolutions of 
$R\# H$ 
$$
\pi: \BB_{R\#\Ho}   \longrightarrow X\sd 
\quad\text{and}\quad 
\iota: X\sd \longrightarrow  \BB_{R\#\Ho}
\, 
$$
with composition the identity: 
$$ 
\pi \, \iota 
= 1_{X}\, . 
$$
Specifically, $\iota = \EZt(\iota_{_R}\ot 1_{\BB_H})$
for $\iota_{_R}$ the standard 
 $R$-bimodule inclusion map $K_R\rightarrow \BB_R$.
\end{thm}
\begin{proof}
  In each degree $n$, $K_n=(K_R)_n$ is spanned as an $R$-bimodule
 by $k\ot \widetilde K_n\ot k$ (see \cref{Koszulterms})
 which is sent into $k\ot \bar R^{\ot n}\ot k$ by $\iota_{_R}$.
 \cref{conversionfree} and its proof
 then imply the result, see also
 \cref{hopfconversion}.
\end{proof}

\vspace{2ex}

\begin{remark}{\em
  We saw in  \cref{conversionfree} that it was not necessary to
  insist upon existence of a projection map $\pi_{_R}:
  \BB_R
    \rightarrow
  K_R $
that is compatible with $\tau$ in order to find chain maps
$\pi: \BB_{R\#\Ho}   \rightarrow X\sd$ and 
$\iota: X\sd \rightarrow  \BB_{R\#\Ho}$
with $\pi \,\iota$ the identity.
Instead 
\cref{bootstrap} allows us in that theorem
  to develop results
  in a more general setting.
  \cref{conversionfree} implies \cref{LastThm} which we use
to verify the next corollary giving 
a chain map of bimodule resolutions of $R\# H$
  $$\pi_{_{R,H}}
  = \BB_R\ott\BB_\Ho \longrightarrow 
K_R \ott \BB_H  $$
  even when a projection map $\pi_{_R}:\BB_R\twoheadrightarrow  
  K_R$ compatible with twisting is not given.
  }
\end{remark} 

\vspace{2ex}

 \begin{cor}
   Let $X\sd$
be the twisted product resolution for 
$R\# \Ho$ that twists the Koszul resolution $K_R$ of $R$
with the reduced bar resolution 
$\BB_\Ho$ of the Hopf algebra $\Ho$. 
There is a graded chain map of bimodule resolutions of 
$R\# H$
  $$\pi_{_{R,H}}
  = \BB_R\ott\BB_\Ho 
\longrightarrow 
K_R \ott \BB_H  $$
 satisfying
$\pi_{_{R,H}}\ (\iota_{_R}\ot 1_{\BB_H}) = 1$.
\end{cor}
\begin{proof}
  We take the maps
  $\pi: \BB_{R\#\Ho}   \rightarrow X\sd$
  and
  $\iota: X\sd \rightarrow  \BB_{R\# H}$
as in \cref{LastThm}
and define $\pi_{_{R,H}}$ as the obvious composition,
  $\pi_{_{R,H}} = \pi \, \EZt$.
  Then
  $\pi_{_{R,H}}\, (\iota_{_R}\ot 1_{\BB_H})$ is
  $\pi\, \EZt\ (\iota_{_R}\ot 1_{\BB_H})
  = \pi\, \iota
  = 1,
  $ 
  the identity on $K_R\ott \BB_H$.
\end{proof}

\vspace{2ex}

\begin{example}\label{ChainMapsGroupActingOnPolyRing}
  {\em
Again consider $k[V]$, the polynomial ring on $V\cong k^n$ with a
linear action of a finite group $G$
as in 
\cref{GroupUnshuffle,GroupExampleAWEZmaps,SymmetricAlgebraGroup}
and the end of~\cref{sec:conversion}.
Note that we do not assume
the characteristic of the underlying field $k$ is coprime to $|G|$.  
Let $X\sd=K\sd\ott \BB_{kG}$
be the twisted product resolution 
that twists the (bimodule) Koszul resolution $K\sd$ of $k[V]$
with the reduced bar resolution
$\BB_{kG}$ of the group algebra $kG$.
We set $\iota_{_{k[V]}}:K\sd\rightarrow \BB_{k[V]}$ to be the standard inclusion
chain map
and $\iota_{_{kG}}$ to be the identity map on $\BB_{kG}$.
By \cref{LastThm}, there are
graded chain maps of bimodule resolutions of
the smash product $k[V]\# G$ 
$$
\pi: \BB_{k[V]\# G}   \longrightarrow X\sd 
\quad\text{and}\quad 
\iota: X\sd \longrightarrow  \BB_{k[V]\#G}
$$
with composition the identity:
$
\pi \, \iota 
= 1_{X}\, . 
$
These chain maps arise from the twisted Alexander-Whitney and
Eilenberg-Zilber maps in a natural way, in contrast to constructions
in the literature that rely on some ad hoc arguments for finite group actions;
for example, compare
with~\cite{SW-quantum}
(see (2.6), (4.2), and the Appendix).
Also see~\cite[Lemma~4.7]{quad}. 
We caution that 
the field has characteristic~0 in~\cite{SW-quantum}
although for the definitions of the given maps, that is not necessary. 
}\end{example}

\vspace{2ex}

\section{Acknowledgments}
We thank Marcelo Aguiar
for engaging conversations on some ideas presented here.
The first author was partially supported by Simons grants 429539 and 949953.
The second author was partially supported by NSF grant 2001163.
 This material is based upon work supported by the National Science Foundation 
under Grant No.~DMS-1928930 and by the Alfred P.~Sloan Foundation under grant 
G-2021-16778, while the second author was in residence at the Simons Laufer 
Mathematical Sciences Institute (formerly MSRI) in Berkeley, California, 
during the Spring 2024 semester.


\end{document}